\documentclass[a4paper,twoside,11pt,reqno]{amsart}
\usepackage{typearea}
\usepackage{url}
\usepackage{hfoldsty}
\usepackage{pstricks,pst-node,pst-tree}
\usepackage[all]{xy}
\usepackage{paralist}
\usepackage{pstricks}
\usepackage{bm}

\usepackage[polutonikogreek,english]{babel}
\newcommand{\Gk}[1]{\selectlanguage{polutonikogreek}#1\selectlanguage{english}}

\author{David Pierce}
\date{\today}
\title{Numbers}
\address{Mathematics Department, Middle East Technical University,
Ankara 06531, Turkey} 
\email{dpierce@metu.edu.tr}
\urladdr{http://metu.edu.tr/~dpierce/}

\newenvironment{sentenceenum}%
{\begin{enumerate}\renewcommand{\labelenumi}{\theenumi.}}{\end{enumerate}}
\newenvironment{clauseenum}%
{\begin{enumerate}}{\end{enumerate}}

\newenvironment{myquote}%
{\begin{quotation}\small}{\end{quotation}}

\renewcommand{\leq}{\leqslant}
\renewcommand{\geq}{\geqslant}
\renewcommand{\phi}{\varphi}
\usepackage{stmaryrd}
\renewcommand{\land}{\mathrel{\binampersand}}
\newcommand{\standard}[1]{\mathbb{#1}}
\newcommand{\N}{\standard N}
\newcommand{\Z}{\standard Z}
\newcommand{\R}{\standard R}
\newcommand{\B}{\standard B}
\newcommand{\universe}{\mathbf V}
\newcommand{\scr}{\mathsf s}
\newcommand{\scrt}{\mathsf t}
\newcommand{\pl}{\mathord+}
\newcommand{\str}[1]{\mathfrak{#1}}
\newcommand{\divides}{\mathrel|}
\newtheorem{theorem}{Theorem}
\newtheorem{lemma}{Lemma}

\newcommand{\lto}{\Rightarrow}
\newcommand{\Exists}[1]{\exists#1\;}

\usepackage{mathrsfs} 
\newcommand{\pow}[1]{\mathscr P(#1)}
\newcommand{\powf}[1]{\mathscr P_{\mathrm f}(#1)}
\usepackage{amssymb}
\renewcommand{\emptyset}{\varnothing}
\usepackage{upgreek}
\newcommand{\vnn}{\upomega}
\newcommand{\pred}[1]{\operatorname{pred}(#1)}
\newcommand{\predk}[1]{\operatorname{pred}_k(#1)}
\newcommand{\predl}[1]{\operatorname{pred}_{\ell}(#1)}
\newcommand{\predo}[1]{\operatorname{pred}_0(#1)}
\newcommand{\sig}{\mathscr S}
\newcommand{\vnns}{\vnn_{\sig}}
\usepackage{bm}
\newcommand{\class}[1]{\bm{#1}}
\newcommand{\Ds}{\mathbf D_{\sig}}
\newcommand{\Dss}{\mathbf D_{\{1,\scr\}}}
\newcommand{\cF}{\class F}
\newcommand{\cL}{\mathbf L}
\newcommand{\cR}{\mathbf R}
\newcommand{\arity}{\operatorname{ar}}

\newcommand{\tm}[1][\sig]{\operatorname{Tm}^0(#1)}
\newcommand{\term}[3]{#1#2_0\cdots#2_{#3-1}}
\newcommand{\ft}{\term Ftn}
\newcommand{\fu}{\term Fun}

\newcommand{\hgt}{\operatorname{hgt}}

\newcommand{\on}{\ensuremath{\mathbf{ON}}}
\newcommand{\zclass}{\mathbf{Zm}}
\newcommand{\wf}{\ensuremath{\mathbf{WF}}}
\newcommand{\ons}{\ensuremath{\mathbf{ON}_{\sig}}}
\usepackage{verbatim}
\newcommand{\included}{\subseteq}
\newcommand{\Zn}[1][n]{\Z/(#1)}
\renewcommand{\setminus}{\smallsetminus}
\newcommand{\chf}[1]{\chi_{#1}}
\newcommand{\tuple}[1]{\bm{#1}}
\newcommand{\gel}{\mathscr G}  

\begin{document}


  \maketitle

\begin{myquote}
  \textsc{There is no more impressive form} of literature than the
  narrative epic poem.  That combination of depth and breadth of
  conception which some have called sublimity has here found a natural
  and adequate expression.  The theory of number is the epic poem of
  mathematics.  The mutual reflection of the two arts will supply a
  sort of explanation in the intellectual dimension of the epic
  quality in both\dots

But the question, what is a number? is an invitation to analyze
counting and to find out what sort of thing makes counting possible.
In poetry I suppose the corresponding question would be, what makes
recounting possible?  The answer, if it were to be complete, would take
us into the most abstract and subtle mathematical thought. But the key
to the problem is the simplest sort of insight.  The same peculiar
combination of simplicity and subtlety is involved in the theory of
narrative, but as everyone knows, insight here belongs to the most
common of common-sense conceptions.

\vspace{0.5\baselineskip}
\mbox{}\hfill---Scott Buchanan, \emph{Poetry and Mathematics}
\cite[p.~64]{Buchanan}
\end{myquote}

\tableofcontents

\setcounter{section}{-1}

\section{Numbers and sets}\label{intro}

The set $\N$ of \textbf{natural numbers} is described by
the so-called \textsl{Peano Axioms}~\cite{Peano}.  These refer to the
following three 
features of $\N$.
\begin{sentenceenum}
  \item
It has a distinguished \textbf{initial element,} denoted by~$1$ or~$0$,
depending on the writer.
\item
It has a distinguished operation of \textbf{succession,} which
converts an element into its \textbf{successor.}  The successor of $x$ is often denoted by
$x+1$, though I shall also use $\scr(x)$ and $x^{\scr}$.
\item
In particular, the operation of succession takes a \emph{single}
argument (unlike addition or multiplication, which takes two arguments). 
\end{sentenceenum}
The class $\universe$ of all \textsl{sets}\footnote{That is,
  \textsl{hereditary} sets, namely sets whose elements are hereditary
  sets \cite[p.~9]{MR85e:03003}.} also has three features.  I shall develop an analogy between the three features of $\N$ listed above and the following three features of $\universe$.
\begin{sentenceenum}
  \item
It has a distinguished element, the empty set, denoted by~$\emptyset$
or~$0$. 
\item
It has one kind or `type' of non-empty set.
This is by the \textbf{Extension Axiom,} according to which a
set is determined solely by its elements (or lack of elements).
If sets are considered as boxes of contents, then
the contents may vary, but the boxes themselves are always cardboard, say,---never plastic or steel.  
\item
Each non-empty
set has a single kind or `grade' of element; the elements are not sorted into
multiple compartments as in an angler's tackle-box.
\end{sentenceenum}

Von Neumann's definition of the \textsl{ordinal numbers}
\cite{von-Neumann} gives us, in particular, the natural numbers.  The
von Neumann definition of the natural numbers can be seen as a
construction of~$\N$
  within~$\universe$ by means of the analogy between $\N$ and
  $\universe$ just suggested.  In general,~$\N$
  can be understood as a \textsl{free algebra} whose \textsl{signature}
  comprises the constant (or nullary function-symbol)~$1$ and the
singulary\footnote{Although the term 
  `unary' is commonly used now in this context, I follow Quine,
  Church~\cite[n.~29, p.~12]{MR18:631a}, and Robinson~\cite[pp.~19,
    23]{MR1373196} in preferring the term   
  \emph{singulary.}  The first five Latin distributive numbers are 
  \emph{\mbox{singul-,} \mbox{bin-,} \mbox{tern-,} \mbox{quatern-,}} and
  \emph{quin-}~\cite{LatinDili}, and these are
  apparently sources for the terms \emph{binary, ternary, quaternary,}
  and \emph{quinary.} 
The Latin cardinals are \emph{un-, du-, tri-, quattuor, quinque.}
(The hyphens stand for variable case-endings.)}
function-symbol~$\scr$.    
Von Neumann's definition gives us the particular free algebra denoted
by $\vnn$, in which~$1$ is
interpreted as $\emptyset$, and~$\scr$ is interpreted as the function
$x\mapsto x\cup\{x\}$.  A similar set-theoretic construction will give us a
free algebra in an arbitrary algebraic signature, as long as there is
a \textsl{type} of set 
for each symbol in the signature, and sets having the type of an $n$-ary
symbol have~$n$ \textsl{grades} of members.  There is a corresponding
  generalization of the notion of an ordinal number as well.
I work out the details of this construction at the end of this
article, in Part~\ref{part:new}.
Meanwhile, in Part~\ref{part:misconceptions}, I consider what I
hold to be misconceptions about numbers.  It was these considerations that led to the whole of this paper.  In
Part~\ref{part:history}, I review the logical development of
numbers from a historical perspective.  

\part{}\label{part:misconceptions}

\section{Numbers and geometry}\label{no-geo}

It is worthwhile to pay close
attention to the fundamental properties of~$\N$.
We may learn these wrongly in school.  We
may be taught---rightly---that~$\N$ has the following properties:
\begin{sentenceenum}
  \item\label{item:ind}
It admits \textbf{proofs by induction:}  every subset that contains $1$ and
that contains $k+1$ whenever it contains $k$ is the whole set.
\item\label{item:wo}
It is \textbf{well-ordered:} every non-empty subset has a least element.
\item\label{item:si}
It admits \textbf{proofs by `complete' or `strong' induction:} every
subset that contains $1$ and that contains $k+1$ whenever it contains $1$,
\dots, $k$ is the whole set. 
\item\label{item:rec}
It admits \textbf{recursive definitions} of functions: on it, there is a
unique function $h$ such that $h(1)=a$ and $h(k+1)=f(h(k))$, where $f$
is a given operation on a set that has a given element $a$.
\end{sentenceenum}
We may also be taught that these properties are \emph{equivalent.}
But they are \emph{not} equivalent.  Indeed, it is `not even
wrong' to say that they are equivalent, since properties~\ref{item:ind}
and~\ref{item:rec} are possible properties of algebras in the signature $\{1,\scr\}$,
while~\ref{item:wo} is a possible property of ordered sets, and~\ref{item:si}
is a possible property of ordered algebras.

A well-ordered set without a greatest element can be made into
an algebra in the signature $\{1,\scr\}$ by letting $1$ be the least element and letting $\scr(x)$ be the least element that is greater than $x$;
but this algebra need not admit induction in the sense
of~\ref{item:ind}.
Indeed, an example of such an algebra is the algebra of ordinals that are less than $\vnn+\vnn$.
Conversely, an algebra in the signature
$\{1,\scr\}$ that admits induction need not admit an
ordering~$<$ such that $x<\scr(x)$: finite cyclic groups provide examples.

The four properties listed above can \emph{become} equivalent under additional 
assumptions; but these assumptions are not usually made explicit.
Such assumptions are made explicit below in Theorems~\ref{thm:rec}
and~\ref{thm:order}.  I present these theorems, not as anything new,
but as being insufficiently recognized.

Perhaps the real problem is a failure to distinguish between what I shall call
\emph{naturalistic} and \emph{axiomatic} approaches to mathematics.  In the former,
mathematical objects exist out in nature; we note some of
their properties, while perhaps 
overlooking other properties that are too obvious to 
mention.  There is nothing wrong with this.
Euclid's \emph{Elements} \cite{MR17:814b,MR1932864} is naturalistic, and
beginning mathematics students would be better served by a
course of reading Euclid than by a course in analytic geometry or in
the abstract concepts of sets, relations, functions, and symbolic logic.

Not all of Euclid's propositions follow logically from his
postulates 
and common notions.  He never claims that they
do.\footnote{Heath~\cite[I, pp.~221~f.]{MR17:814b} discusses the
  possibility that the the so-called Common Notions in the
  \emph{Elements} are a later interpolation.}
His Proposition I.4
establishes the `side-angle-side' condition for congruence of
triangles, but the proof relies on `applying' one triangle to
another.  Perhaps this `application' is alluded to in his
Common Notion~4: `Things which coincide with one another are equal to
one another.'  
Suppose indeed that in
triangles $ABC$ and $DEF$ in the figure below, 
\begin{figure*}[ht]
  \begin{pspicture}(-0.5,-0.5)(5.5,3.5)
    \psline(0,0)(4,0)(5,3)(0,0)
\uput[d](0,0){$A$}
\uput[d](4,0){$B$}
\uput[u](5,3){$C$}
  \end{pspicture}
  \begin{pspicture}(-0.5,-0.5)(5.5,3.5)
    \psline(0,0)(4,0)(5,3)(0,0)
\uput[d](0,0){$D$}
\uput[d](4,0){$E$}
\uput[u](5,3){$F$}
  \end{pspicture}
\end{figure*}
sides $AB$ and $DE$ are equal,
and $AC$ and $DF$ are equal, and angles $BAC$ and $EDF$ are equal.  If
the one triangle is `applied' to the other, $A$ to $D$ and $AB$ to
$DE$, then $B$ and $E$ will coincide, and also $C$ and $F$, by a sort
of converse to Common Notion 4 that Euclid does not make explicit:
things that are equal can be made to coincide.\footnote{Such
  coinciding, in the case of plane figures, may involve cutting and
  rearranging, as in the series of propositions that begins with
  I.35: parallelograms with the same base in the same parallels are
  equal.
  \begin{center}
\mbox{}\hfill
    \begin{pspicture}(-0.3,0)(1.5,1.2)
    \psline   (0,0)(-0.3,1)
    \psline (1.2,1) (1.5,0)
    \psline   (0,0) (0.1,1)
    \psline (1.6,1) (1.5,0)
    \psline   (0,0) (1.5,0)
    \psline(-0.3,1) (1.6,1)
  \end{pspicture}
\hfill
  \begin{pspicture}(-0.9,0)(2.7,1.2)
    \psline   (0,0)(-0.9,1)
    \psline (1.5,0)( 0.6,1)
    \psline   (0,0) (1.2,1)
    \psline (1.5,0) (2.7,1)
    \psline   (0,0) (1.5,0)
    \psline(-0.9,1) (2.7,1)
    \end{pspicture}
\hfill\mbox{}
  \end{center}}  
Then $BC$ and $EF$ 
coincide; indeed, as observed in a remark that is bracketed
in~\cite{MR17:814b}, but omitted in \cite{MR1932864}, if the two
straight lines do not coincide, then `two straight 
lines will enclose a space: which is impossible.'  Again, this is a
principle so obvious as not to be worth treating as a postulate.

Without any explicit postulates,
Euclid proves his number-theoretic propositions, 
from the Euclidean algorithm (VII.1--2) 
to the perfectness of the product of
$2^n$ and $1+2+4+\dotsb+2^n$ if the latter is prime (IX.36).

Hilbert's approach in \emph{The Foundations of Geometry}
\cite{MR0116216} is axiomatic.  Hilbert himself seems to see his
project as the same as Euclid's:
\begin{myquote}
  Geometry, like arithmetic, requires for its logical
development only a small number of simple,
fundamental principles. These fundamental principles
are called the axioms of geometry. The choice
of the axioms and the investigation of their relations
to one another is a problem which, since the time of
Euclid, has been discussed in numerous excellent
memoirs to be found in the mathematical literature.
This problem is tantamount to the logical analysis of
our intuition of space.

The following investigation is a new attempt to
choose for geometry a \emph{simple} and \emph{complete} set of
\emph{independent} 
axioms and to deduce from these the most important
geometrical theorems in such a manner as to
bring out as clearly as possible the significance of the
different groups of axioms and the scope of the conclusions
to be derived from the individual axioms.
\end{myquote}
However, it is not at all clear that Euclid is concerned with
relations among axioms.  Proving the \emph{independence}
of one axiom from others usually means showing that the latter have a model in
which the former is false.  Hilbert can do this, but Euclid makes no
suggestion of different models of
his postulates.  Euclid makes observations and deductions about
the world---the 
\Gk{g~h} of \Gk{gewmetr'ia}.\footnote{Euclid does not discuss his
  work, and in particular he does not use the word \Gk{gewmetr'ia}.
  Herodotus does use the word, a century or two earlier:
  \begin{quotation}
    For this cause Egypt was intersected [by canals].  This king
    [Sesostris] moreover (so they said) divided the country among all
    the Egyptians by giving each an equal square parcel of land, and
    made this his source of revenue, appointing the payment of a
    yearly tax.  And any man who was robbed by the river of a part of
    his land would come to Sesostris and declare what had befallen
    him; then the king would send men to look into it and measure the
    space by which the land was diminished, so that thereafter it
    should pay in proportion to the tax originally imposed.  From
    this, to my way of thinking, the Greeks learnt the \emph{art of
    measuring land} [\Gk{gewmetr'ih}]; the sunclock and the sundial,
    and the twelve 
    divisions of the day, came to Hellas not from Egypt but from
    Babylonia.\hfill \mbox{\cite[2.109]{Herodotus-Loeb}}
  \end{quotation}
For \Gk{gewmetr'ih} (the Ionic form of the word) here, other translators, as
    \cite{Herodotus-Penguin}, just use `geometry'.  It is too simple
    to say that geometry developed from surveying, since it was not
    necessary for geometry as we know it to develop at all.  On the
    other hand, if our experience of the world came entirely from
    gazing at the heavens, and if we nonetheless developed a geometry,
    then this would probably be what we in fact call
    spherical geometry.} 
By constrast, while Hilbert's axioms are
obviously true in the world of Euclid, Hilbert nonetheless feels the
need to establish their consistency by constructing a model.  His
model is an `analytic' one (in the sense of analytic geometry): it is based
on the field obtained from the number $1$ by closing under $+$, $-$,
$\times$, $\div$, and $x\mapsto\sqrt{1+x^2}$.  For
Hilbert, it seems, such a field is more real than Euclidean space, even though
our intuition for fields, and especially for taking square roots in
them, comes from geometry. 

The approach to~$\N$ in which
properties~\ref{item:ind}--\ref{item:rec} above are `equivalent' is
not quite 
naturalistic, not quite axiomatic.  Considered as axioms in the sense
of Hilbert,
the properties are not meaningfully described as equivalent.
But if the properties are
to be understood just as properties of the numbers that we grew up
counting, then it is also
meaningless to say that the properties are equivalent: they are just
properties of those numbers.  

An axiomatic approach to mathematics may lead to a constructive
approach, whereby we \emph{create} our objects of study.
Hilbert could have ignored his axioms and concentrated on his model of
them, although this did not happen to be his purpose.  Regarding~$\N$,
we may or may not wish ultimately to take a constructive approach,
creating a model of the Peano Axioms by, for example, the method of
von Neumann.  If we do take this approach, then we presumably have
recognized something more fundamental than the natural numbers:
perhaps sets, which we may have approached naturalistically or
axiomatically.  But then the construction of the natural
numbers may add to our understanding of, and confidence in, these
sets.  This is indeed how I see the von Neumann construction, along with the generalization in Part~\ref{part:new} below.

\section{Numbers, quasi-axiomatically}\label{sect:nqa}

Through Spivak's \emph{Calculus} \cite{Spivak-Calc}, I had
my first encounters with serious mathematics.  Because this book
is worthy of scrutiny, I examine here its treatment of
foundational matters.  In the beginning, Spivak takes the
real numbers as being out in nature.  He just calls 
them `numbers,' and in his first chapter he identifies twelve of their
properties.  These properties will, in his chapter 27, be recognized as
the axioms for ordered 
fields.  Meanwhile, in some proofs in his chapter 1, Spivak uses
assumptions that he goes on to 
identify as having been unjustified.  For example, to
prove that, if a product is zero, then one of the factors must be
zero, Spivak observes that, if $a\cdot b=0$, and $a\neq0$, then
$a^{-1}\cdot(a\cdot b)=0$.  Thus he
tacitly assumes that zero times anything is zero; on the
next page, he comes clean and proves this.  I find no fault here.

The difficulties for me arise in Spivak's chapter 2.  Here the various `sorts'
of numbers are introduced, starting with the natural numbers.  The
four properties of the natural numbers listed above in \S\ref{no-geo}
are stated, and it 
is said that, from each of the first three properties, the others can
be proved. 
For example, Spivak proves by induction that a set of natural numbers
without a least
element must be empty:  $1$ cannot be in the set, since
otherwise it would be the least element; and for the same reason, if
none of the numbers $1$, \dots,~$k$ is in the set, then $k+1$ is not
in the set.

The problem is that this argument uses more than induction: it uses
that the set of natural numbers is ordered so that
$n\leq k$ or $k+1\leq n$ for all elements $k$ and $n$.  The existence
of this ordering does not follow by
induction alone.  It \emph{does} follow, if one has the natural
numbers as real numbers.  Again, for Spivak, the real numbers are
just `numbers', so the natural numbers must be among these; but Spivak
does not emphasize this in the text.  In
exercise~25 at the end of chapter 2, an \textsl{inductive set} is defined
as a set of numbers that contains $1$ and is closed under addition of $1$.
The reader is invited to show that the set of numbers common to all
inductive sets is inductive.  It could be, but is not, made clear at
this point that the properties called~\ref{item:ind}
and~\ref{item:rec} in \S\ref{no-geo} above are
equivalent properties of \emph{inductive sets.}  

Spivak's \emph{Calculus} was a reference text and a source of
exercises for a two-year
high-school course taught by one Donald J. Brown.\footnote{This
  was at St.\ Albans School in Washington, D.C., 1981--3.} 
The main text for the course was 
notes that Brown wrote and distributed, either as typed and mimeographed
pages, or more usually as writing on the blackboard that we students
copied down. 
Brown's notes are in some ways more 
formally rigorous than Spivak.  From the beginning, they
give the real numbers explicitly as composing the complete ordered
field called $\R$. And yet the notes present the natural numbers as a set
$\{1,2,3,\dots\}$ for which the `fundamental axiom' is the `Well
Ordering Principle'.  From this, the `Principle of Mathematical
Induction' is obtained as a theorem; it is left as an exercise to
obtain the Well Ordering Principle as a theorem from the Principle of
Induction.  Here again, proofs must rely on hidden
assumptions.  At the relevant point in my copy of Brown's
text, I find the following remark, apparently added by me during the
course:  `If the given ``definition'' of~$\N$ means anything, it
must be the equivalent of the Principle of Mathematical Induction.
Thus the P. of M.I. need not be a theorem, and the W.O.P. need not be
an axiom.'  


\part{}\label{part:history}

\section{Numbers, recursively}\label{sect:rec}

Let us forget about $\R$ and try to start from scratch.
We should distinguish clearly between {recursion} from {induction.}
\textbf{Recursion} is a method of \emph{defining} sets and functions;
\textbf{induction} is a method of \emph{proving} that one set is equal
to another.\footnote{One text that makes the distinction clear is
  Enderton \cite[\S 1.2, pp.~22-30]{MR0337470}.}
We can try to understand~$\N$
as having the following \textbf{recursive definition:}
\begin{clauseenum}
  \item\label{N1}
$\N$ contains $1$;
\item\label{N2}
if~$\N$ contains $x$, then~$\N$ contains $x^{\scr}$.
\end{clauseenum}
Since this is explicitly a definition, there is really no need for a
third condition, though it is sometimes given:
\begin{clauseenum}\setcounter{enumi}2
  \item
$\N$ contains nothing but what it is required to contain by~\eqref{N1}
and~\eqref{N2}. 
\end{clauseenum}
Again, I hold this third condition to be redundant for the same
reason that the words \emph{only if} are redundant when, for example,
we define a natural number as \textsl{prime} if and only if it has just
two distinct divisors. 

The recursive definition of~$\N$ means no more nor less than that, if
$A$ is a subset of~$\N$ that contains $1$ and that contains the
successors of all of its elements, then $A=\N$.
In short, the definition means that~$\N$ \textbf{admits proofs by
  induction.}  This can also be expressed
by writing
\begin{equation*}
  \N=\{1,1^{\scr},1^{\scr}{}^{\scr},\dots\}.
\end{equation*}
The recursive definition of~$\N$ does not bring~$\N$ into existence.
Rather, it picks out~$\N$ from among some things that are assumed
to exist already.  The definition assumes that there is \emph{some}
set that contains $1$ and is closed under $\scr$.  Then~$\N$ is the
smallest such set, namely the intersection of the collection of such
sets. 

Definitions such as the one just given have been found objectionable
for being `impredicative.'  Zermelo
\cite[p.~191, n.~8]{Zermelo-new-proof}
reports that Poincar\'e had such an objection to Zermelo's proof of
the Schroeder--Bernstein Theorem.  Let us have a look at this proof.

Suppose $A$, $B$, and $C$ are sets such that $A\included B\included
C$, and $f$ is a bijection from $C$
onto~$A$.  We aim to find a bijection $g$ from $B$ onto $A$.  If there
is any hope of finding such~$g$, then probably we should have
\begin{equation}\label{gBA}
  g(x)=
  \begin{cases}
    f(x),&\text{ if }x\in D,\\
x,&\text{ if } x\in B\setminus D,
  \end{cases}
\end{equation}
for some subset $D$ of $B$.  In this case, since the range of $g$ is
to be (a subset of) $A$, we must have in particular $B\setminus D\included A$ and so
$B\setminus A\included D$.  Injectivity of $g$ will follow from having
$f[D]\included D$.  Let $D$ now be the \emph{smallest} $X$ such
that 
\begin{equation*}
  (B\setminus A)\cup f[X]\included X.
\end{equation*}
That is, let $D$ be the intersection of the set of all such sets $X$.
We obtain $g$ as desired.  Indeed, when $D$ is such, then $(B\setminus
A)\cup f[D]=D$.  Since also $F[D]\included A$, we have that $B\setminus A$
and $f[D]$ are disjoint.  Therefore $B$ is the disjoint union of
$B\setminus A$, $f[D]$, and $B\setminus D$, as in the figure below.
We conclude $(B\setminus D)\cup f[D]=A$, so $g$ is indeed surjective
onto~$A$.  
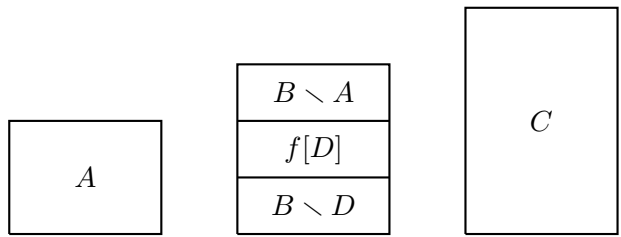
\begin{figure*}[ht]
  \begin{pspicture}(8,3)
    \psline(6,0)(8,0)(8,3)(6,3)(6,0)
\rput(7,1.5){$C$}
    \psline(3,0)(5,0)(5,2.25)(3,2.25)(3,0)
\rput(4,1.875){$B\setminus A$}
    \psline(3,1.5)(5,1.5)
\rput(4,1.125){$f[D]$}
    \psline(3,0.75)(5,0.75)
\rput(4,0.375){$B\setminus D$}
    \psline(0,0)(2,0)(2,1.5)(0,1.5)(0,0)
\rput(1,0.75){$A$}
  \end{pspicture}
\end{figure*}

Such is the proof in the article \cite[p.~208]{Zermelo-invest} of
Zermelo giving his axioms of set-theory, although Zermelo does not
give the argument in the heuristic\footnote{The historically correct
  word is \emph{analytic,} in view of remarks like this of Pappus:
  \begin{quote}
    Now \emph{analysis} [\Gk{>an'alusic}] is a method of taking that
    which is sought as though it were admitted and passing from it
    through its consequences in order to something which is admitted
    as a result of synthesis [\Gk{sunj'esic}]; for in analysis we
    suppose that which is sought to be already done, and we inquire
    what it is from which this comes about\dots and such a method we
    call analysis, as being a reverse solution [\Gk{>an'apalin
        l'usic}].\hfill\cite[p.~601]{MR13:419b} 
  \end{quote}
The style of geometry pioneered by Descartes is not so much analytic
as \emph{algebraic.}}
 style that I have used.  The
argument might be said to
assume the existence of what it purports to establish the existence
\emph{of,} since $D$ belongs to the set of which it is the intersection.  

For an alternative heuristic argument, again suppose $g$ is as in~\eqref{gBA}, so that
\begin{equation}\label{eqn:gB}
g[B]=f[D]\cup(B\setminus D).
\end{equation}
Say $g[B]=A$.
Then in particular $g[B]\included A$, so by~\eqref{eqn:gB} we have $B\setminus D\included A$, hence 
\begin{equation}\label{eqn:BAD}
B\setminus A\included D.
\end{equation}
Also, $A\included g[B]$, but by~\eqref{eqn:gB} we have $g[B]\included f[B]\cup(B\setminus D)$, hence 
\begin{equation}\label{eqn:AfB}
A\setminus f[B]\included B\setminus D.
\end{equation}
Since $(B\setminus A)\cup(A\setminus f[B])=B\setminus f[B]$, we now have 
\begin{align*}
  g[B\setminus f[B]]
&=g[(B\setminus A)\cup(A\setminus f[B])]&&\\
&=g[B\setminus A]\cup g[A\setminus f[B]]&&\\
&=f[B\setminus A]\cup (A\setminus f[B])&&\text{[by~\eqref{eqn:BAD} and~\eqref{eqn:AfB}]}\\
&=(f[B]\setminus f[A])\cup (A\setminus f[B])
=A\setminus f[A].&&
\end{align*}
Assuming $g$ is injective, we conclude
$g[f[B]]=f[A]$.  Now we can proceed as before, but with $f[B]$
and $f[A]$ in place of $B$ and $A$.  In particular, we replace~\eqref{eqn:gB} with
\begin{equation*}
g[f[B]]=f[D\cap f[B]]\cup(f[B]\setminus D)
\end{equation*}
and obtain $f[B]\setminus f[A]\included D$, and so forth.
We find ultimately
\begin{equation*}
  \bigcup\{B\setminus A,f[B]\setminus f[A],f[f[B]]\setminus
  f[f[A]],\dots\}\included D. 
\end{equation*}
Conversely, to ensure that $g[B]=A$, it suffices to let $D$ be this
union.  
Indeed, this is the same $D$ found in Zermelo's argument above, only
now it is found through a \textbf{recursive} procedure, as
$\bigcup\{B_n\setminus A_n\colon n\in\N\}$, where
\begin{align*}
  B_1&=B,& A_1&=A,& B_{\scr(n)}&=f[B_n],& A_{\scr(n)}&=f[A_n].
\end{align*}
Is this better than finding $D$ as the intersection of a set that contains $D$?
Well, it assumes the existence of $\N$, whose recursive definition was accused of being impredicative.  What is worse, the existence of the function
$n\mapsto(A_n,B_n)$ is not justified merely by the
possibility of defining $\N$ itself recursively as we did above. 
For this reason, I prefer Zermelo's argument.

I propose to refer to any set\footnote{Later the possibility of using a proper class here will be considered.} that has a distinguished element and a
distinguished singulary
operation as an \textbf{iterative structure.}\footnote{Stoll \cite[\S 2.1,
  p.~58]{MR83e:04002} 
uses the term `unary system'.}  The distinguished element
can be called~$1$, and the operation, $\scr$; or they can be called
$1^{\str A}$ and 
$\scr^{\str A}$, to avoid ambiguity, if the structure itself is $\str
A$, but there are other iterative structures around.\footnote{This is like calling me David, unless there is another David present,
in which case I may become David Pierce.}  
Then an iterative structure is just an algebra in the signature $\{1,\scr\}$.
An iterative structure
meets minimal requirements for repeated activity:
it gives us something---$\scr$---to do, and something---$1$---to do it
to. 

Every iterative structure has a unique smallest substructure, which is
the intersection of the collection of all substructures;\footnote{This assumes that the iterative structure is based on a set, rather than a proper class.} this intersection
therefore admits induction and can be recursively defined.  In
particular,~$\N$ is such an intersection; but~$\N$
has properties beyond admitting induction.  In particular, if $\str A$ is
another iterative structure, then there is a unique homomorphism $h$
from~$\N$ to $\str A$, given by the \textbf{recursive definition}
\begin{clauseenum}
  \item
$h(1)=1$, that is, $h(1^{\N})=1^{\str A}$;
\item
$h(x^{\scr})=h(x)^{\scr}$, that is, $h(\scr^{\N}(x))=\scr^{\str
  A}(h(x))$. 
\end{clauseenum}
In short,
\begin{equation}\label{h}
 h= \{(1,1),(1^{\scr},1^{\scr}),(1^{\scr}{}^{\scr},1^{\scr}{}^{\scr}),\dots\},
\end{equation}
where the left-hand entries in the ordered pairs are from~$\N$; the
right-hand, from $A$.

We have now merely \emph{asserted} the possibility of defining functions recursively on $\N$.
Arithmetic follows from this assertion: the binary
operations of addition, multiplication, and exponentiation can be
defined recursively in their right-hand arguments.  First we have
\begin{align}\label{eqn:rec+}
  m+1&=m^{\scr},& m+x^{\scr}&=(m+x)^{\scr},
\end{align}
that is, the function $x\mapsto m+x$ is the unique homomorphism from $(\N,1,\scr)$ to
$(\N,m^{\scr},\scr)$. 
Then
$\scr$ is $x\mapsto x+1$, and we
continue with the definitions
\begin{align}\label{rec.}
  m\cdot1&=m,& m\cdot(x+1)&=m\cdot x+m;\\\label{exp}
  m^1&=m,& m^{x+1}&=m^x\cdot m;
\end{align}
that is, the functions $x\mapsto m\cdot x$ and $x\mapsto m^x$ are the unique
homomorphisms from $(\N,1,\scr)$ to $(\N,m,x\mapsto x+m)$ and
$(\N,m,x\mapsto x\cdot m)$ respectively.

Skolem \cite{Skolem-foundations} develops arithmetic by recursion and
induction alone, though without discussing the logical relations
between these two methods and without mentioning homomorphisms as
such.  Skolem's purpose is to avoid logical quantifiers.  Instead of
defining $m<n$ to mean $m+x=n$ for \emph{some} $x$, he defines it by
saying that $m<1$ never, and $m<k+1$ if $m<k$ or $m=k$.  To see this
definition in terms of homomorphisms, let us denote the set $\{0,1\}$
by $\B$ for Boole.\footnote{From Boole himself \cite[p.~41]{Boole}:
  `We have seen that the symbols of Logic
  are subject to the special law,
  \begin{equation*}
    x^2=x.
  \end{equation*}
Now of the symbols of Number there are but two, viz.\ 0 and 1, which
 are subject to the same formal law.'}
Every binary
relation $R$ on~$\N$ corresponds to a binary \textsl{characteristic
  function} $\chf R$ from~$\N$ into $\B$ given by
\begin{equation*}
  \chf R(x,y)=1\iff x\mathrel Ry.
\end{equation*}
We are defining the function $x\mapsto\chf<(m,x)$ from $\N$ to $\B$,
which is described by the table below.
\begin{table*}[h]
\begin{equation*}
  \begin{array}{c||c|c|c|c|c|c|c|c}
        x &1&2&\dots&m-1&m&m+1&m+2&\dots\\\hline
\chf<(m,x)&0&0&\dots&  0&0&  1&  1&\dots
  \end{array}
\end{equation*}
\end{table*}
Unless $m=1$, the function is \emph{not} a homomorphism from $(\N,1,\scr)$
into $(\B,0,t)$ for any choice of~$t$, simply because there is no
operation $t$ on $\B$ that takes each entry in the bottom row of the
table to the next entry.  However, there is an operation on
$\N\times\B$ that takes each \emph{column} of the table to the next
column, namely
$(x,y)\mapsto\Bigl(x+1,\max\bigl(y,\chf=(m,x)\bigr)\Bigr)$.  Call this
operation $t$; then (assuming $m\neq1$) the function
$x\mapsto\bigl(x,\chf<(m,x)\bigr)$ is the unique homomorphism from
$(\N,1,\scr)$ to $\bigl(\N\times\B,(1,0),t\bigr)$.

If $f$ is an arbitrary function from $\N$ to an arbitrary set $A$,
then the function $x\mapsto\bigl(x,f(x)\bigr)$ is a homomorphism from
$(\N,1,\scr)$ to $\Bigl(\N\times
A,\bigl(1,f(1)\bigr),(x,y)\mapsto\bigl(x+1,f(x+1)\bigr)\Bigr)$.  But
to \emph{define} $f$ this way would be not only impredicative, but circular.
The operations on
$\N$ that G\"odel~\cite{Goedel-incompl} calls \textbf{recursive} are not
defined as homomorphisms.  Rather, the collection of recursive operations can itself be defined \textbf{recursively} as follows (and here I follow G\"odel in letting the initial element of $\N$ be $0$):
\begin{clauseenum}
  \item
constant operations and $\scr$ are recursive;
\item
compositions of recursive operations are recursive;
\item
if $\psi$ is an $(n-1)$-ary, and $\mu$ an $(n+1)$-ary, recursive operation
for some positive $n$, then the $n$-ary operation $\phi$ is recursive,
where, for each $n$-tuple $\tuple a$ of natural numbers, the function
$x\mapsto\bigl(x,\phi(x,\tuple a)\bigr)$ is the unique homomorphism from $(\N,0,\scr)$
to $\Bigl(\N\times\N,\bigl(0,\psi(\tuple a)\bigr),
(y,z)\mapsto\bigl(y+1,\mu(y,z,\tuple a)\bigr)\Bigr)$.
\end{clauseenum}
Thus the recursive operations compose a sort of algebra which is
closed both under composition and under the operation that produces
$\phi$ from $\psi$ and $\mu$.

\section{Numbers, axiomatically}\label{sect:ax}

\emph{Why} does $\N$ admit recursive definitions of functions?  The function
$h$ in~\eqref h can be understood as the smallest subset
of $\N\times A$ that contains $(1,1)$ and is closed under the operation
$(x,y)\mapsto(x^{\scr},y^{\scr})$.  This set is a \emph{relation} $R$
from~$\N$ to $A$, and it follows by induction that, for
each $x$ in~$\N$,
there is at least one $y$ in $A$ such that $x\mathrel Ry$.
If this $y$ is always unique, then $R$ is the desired homomorphism $h$.
However, uniqueness of $y$ does not follow by induction alone: it is necessary, as well as sufficient, that $1$ be not a successor in $\N$ and that immediate predecessors, when they do exist, be unique.  This is Theorem~\ref{thm:rec} below.

Dedekind~\cite[II (130), pp.~88--90]{MR0159773} observes the sufficiency of the additional conditions, along with the necessity of \emph{some} additional condition.
Indeed, let us say that an arbitrary iterative structure $\str A$
\textbf{admits recursion} if,
for every iterative structure~$\str B$, there is a unique homomorphism
from $\str A$ to $\str B$.  In particular, $\N$ admits recursion.
If also~$\str A$ admits recursion, then the
homomorphism from $\str A$ to~$\N$ 
must be the inverse of the homomorphism from $\N$ to $\str A$.  Indeed, the composition (in either sense) of the two homomorphisms is a homomorphism, but so is the identity, hence these must be equal.  Therefore
$\str A$ and $\N$ are isomorphic, and $\str A$
admits induction, simply because~$\N$ does.  However, as Dedekind notes, the
converse does not follow; $\str A$ may admit induction, but not recursion.
Indeed, for any natural
number~$n$, the quotient $\Zn$ can be
considered as an iterative structure whose initial element is $0$ and
in which the successor of $x$ is $x+1$.  Then $\Zn$
admits induction; but there is no homomorphism from $\Zn$
to $\Zn[m]$ unless $m\divides n$.

As Landau~\cite[Thms 4 \&\ 28; see also pp.~ix.~f.]{MR12:397m}\ shows
implicitly, the recursive
definitions of addition and multiplication in~\eqref{eqn:rec+}
and~\eqref{rec.} are in fact valid on arbitrary
iterative structures that admit induction.  Such structures may be
finite, as in the figure below (where $n$ may be $1$, and $m$ may be~$0$).
\begin{figure*}[ht]
    \begin{pspicture}(0,-1)(8.5,2.5)
    \psset{linewidth=0.5mm%
}
    \psline(0,2)(7,2)
    \psline{->}(5.5,0.5)(5.5,1.7)
    \psarc(7,0.5){1.5}{180}{450}
    \psdots(0,2)(1,2)(2,2)(5.5,2)(6.5,2)(5.5,1)
    \uput[u](0,2.1){$1$}
    \uput[u](1,2.1){$2$}
    \uput[u](2,2.1){$3$}
    \uput[u](5.5,2.1){$n$}
    \uput[ur](6.5,2.1){$n+1$}
    \uput[l](5.4,1){$n+m-1$}
  \end{pspicture}
\end{figure*}
Henkin~\cite{MR0120156}
makes the point explicit, while observing that the recursive definition of
exponentiation fails in some such
structures.  Indeed, since for Henkin the initial element of~$\N$ is
$0$, his definition of exponentiation is not~\eqref{exp}, but
\begin{align*}
  x^0&=\scr(0),& x^{\scr(u)}&=x^u\cdot x,
\end{align*}
which always fails in $\Zn$ in case $n>1$,
simply because the definition requires both $0^0=1$ and $0^n=0$.  
However, if we keep $1$ as the initial element of~$\N$, then the
following curiosity arises.

\begin{theorem}
On $\Zn$, the recursive definition~\eqref{exp} of exponentiation
is valid if and only if $n\in\{1,2,6,42,1806\}$.
\end{theorem}

\begin{proof}
  The desired $n$ are just such that
  \begin{equation*}
    x^{n+1}\equiv x\pmod n.
  \end{equation*}
Such $n$ are found by Zagier~\cite{Zagier}, and earlier by Dyer-Bennet
\cite{MR0001234}, as I learned by entering the sequence $1$, $2$, $6$, $42$,
$1806$ at~\cite{OEIS}. 
Alternatively,
such $n$ must be squarefree, since if $p$ is prime and $p^2\divides n$, then
$(n/p)^2\equiv0\pmod n$, although $n/p\not\equiv0\pmod n$.  So we want to
find the squarefree~$n$ whose
prime factors $p$ are such that $x^{n+1}\equiv x\pmod p$, that is,
either $p\divides x$ or $x^n\equiv1\pmod p$.  Since $x$ here might be a
primitive root of $p$, we just want $n$ such that
$p-1\divides n$ whenever $p$ is a prime factor of $n$.
Let us refer to a prime $p$ as \textbf{good} if
$p-1$ is squarefree and all prime factors of $p-1$ are good.  
Then all prime factors of $n$ must be good.
This definition of good primes can be understood as being \textbf{recursive.}  Indeed, the empty
set consists of good 
primes.  If $A$ is a finite set of good primes,
then every prime of the form
\begin{equation*}
 1+ \prod_{p\in B}p,
\end{equation*}
where $B\included A$, is good.  Then let $A'$ comprise these primes,
along with the primes in $A$.  We have a recursively defined
function $x\mapsto A_x$ on~$\N$, where $A_1=\emptyset$ and
$A_{x+1}=A_x{}'$.  The set of good primes is the union of the sets $A_x$.
We compute
\begin{align*}
A_2&=\{2\},
&A_3&=\{2,3\},
&A_4&=\{2,3,7\},
&A_5&=\{2,3,7,43\},
\end{align*}
but then the sequence stops growing, since
\begin{align*}
2\cdot 43+1&=87=3\cdot29;&2\cdot 7\cdot 43+1&=603=3^2\cdot 67;\\
  2\cdot 3\cdot 43+1&=259=7\cdot 37;&
2\cdot 3\cdot 7\cdot 43+1&=1807=13\cdot 139.
\end{align*}
So the set of good primes is $\{2,3,7,43\}$.  In this set, we have
\begin{equation*}
  p<q\iff p\divides q-1.
\end{equation*}
Hence the set of desired $n$ is $\{1,\;2,\;2\cdot 3,\;2\cdot 3\cdot
7,\;2\cdot 3\cdot 7\cdot 43\}$, which is as claimed.
\end{proof}

The sequence $(2,3,7,43)$ of good primes is the beginning of
the sequence $(a_1,a_2,\dots)$, where $a_{n+1}=a_1\cdot a_2\dotsm a_n+1$.
This sequence arises from what 
Mazur~\cite{Mazur-Th} calls the \emph{self-proving} formulation of
Proposition IX.20 of Euclid's \emph{Elements.}  Euclid's formulation,
in Heath's translation, is, `Prime numbers are more than any assigned
multitude of prime numbers.'\footnote{\Gk{O<i pr~wtoi >arijmoi
    ple'iouc e>is`i pant`oc to~u protej'entoc pl'hjouc pr'wtwn
    >arijm~wn} \cite{Perseus}.
Euclid says not that there are
  \emph{infinitely many} prime numbers, but that there are more than
  can be told.}  In these terms, the self-proving formulation is that,
if a multitude of primes be assigned, then the product of its members,
plus one, is a number whose prime factors---of which there is at least
one---are not in the assigned multitude. 

Again, as an iterative structure, $\N$ admits recursive definition of
functions because of the theorem below. 
The theorem should be standard.  
Dedekind \cite[II (126), p.~85]{MR0159773} proves the reverse
implication, as does, more recently, 
Stoll \cite[ch.~2, Thm 1.2, p.~61]{MR83e:04002}.  Mac Lane and
  Birkhoff \cite{MR0214415} mention the theorem, apparently
  attributing the forward direction to Lawvere: they refer to
  admission of recursion by an 
  iterative structure as the Peano--Lawvere Axiom.
  (Lawvere and Rosebrugh \cite{MR1965482} call it the Dedekind--Peano
  Axiom.) 

\begin{theorem}[Recursion]\label{thm:rec}
For an iterative structure to admit recursion, it is sufficient and
necessary that 
\begin{clauseenum}
\item
it admit induction,
  \item
$1$ be not a successor, and
\item
every successor have a unique immediate predecessor.\qedhere
\end{clauseenum}
\end{theorem}

An iterative structure meeting the three conditions listed in the
theorem is what Dedekind~\cite[II
  (71), p.~67]{MR0159773} calls a \textbf{simply infinite system;} it is what
is axiomatized by the so-called Peano Axioms~\cite{Peano}.  Peano
seems to assume that the Recursion Theorem is 
obvious, or at least that recursive definitions are obviously justified by
induction alone.  They are not, as Landau points out
\cite[pp.~ix--x]{MR12:397m}; but the confusion continues to be
made.\footnote{Indeed, Mac Lane and Birkhoff \cite{MR0214415} seem to
  invite the student to make the confusion.  They give the Peano
  Axioms and then immediately use \emph{recursion} to define the
  iterates of a permutation of a set.  Only later is the equivalence
  of recursion to the axioms mentioned.
Another writer who invites confusion is Burris \cite[Appendix B,
    p.~391]{Burris}: he states the Peano Axioms and then
  immediately defines addition recursively without justification,
  although he claims to be following the presentation of
  Dedekind.} 

To establish the \emph{sufficiency} of the conditions given in the
theorem, we can use them as suggested at the head of this section.
Namely, we can show that the unique homomorphism, given in~\eqref{h}, from $\N$ to an
arbitrary iterative structure $\str A$ is the intersection $h$ of the
set of all \emph{relations} from $\N$ to $A$ that contain $(1,1)$ and
are closed under $(x,y)\mapsto(x^{\scr},y^{\scr})$.  This requires
assuming that $\N$ itself is indeed a set, as opposed to a proper
class.  
Without this assumption, one can obtain $h$ as the union of the set
comprising every relation $R$ from $\N$ to $A$ such that, if
$x\mathrel Ry$, then either $(x,y)=(1,1)$ or else
$(x,y)=(u^{\scr},v^{\scr})$ for some $(u,v)$ such that $u\mathrel
Rv$. 

Parallel to the two ways of obtaining $h$ as in~\eqref{h}, there are two ways to obtain $\N$ itself; these are considered in the next section.

\section{Numbers, constructively}\label{sect:nc}

The next two sections have aspects of an historical review of
set-theory.  However, I refer only to articles and books that I have
direct access to (mainly, though not exclusively, in van 
Heijenoort's anthology \cite{MR1890980}); therefore my review is
incomplete, as the references in Fraenkel \emph{et
  al.}~\cite{MR0345816} and in Levy \cite{MR1924429} make me aware.  

In a letter to a skeptic called Keferstein, Dedekind
\cite{Dedekind-letter} explains his work as an answer to the question:
\begin{myquote}
  What are the mutually independent fundamental properties of the
  sequence~$\N$, that is, those properties that are not derivable from
  one another, but from which all others follow?\footnote{I take the
  liberty of writing~$\N$ where van Heijenoort's text has simply $N$.}
\end{myquote}
Having discovered these properties and made them the defining
properties of a simply infinite system, Dedekind feels the need to
ask,
\begin{myquote}
  does such a system \emph{exist} at all in the realm of our ideas?
  Without a logical proof of existence it would always remain doubtful
  whether the notion of such a system might not perhaps contain
  internal contradictions.
\end{myquote}
Dedekind finds his proof in `the totality $S$ of all things, which can
be objects of my thought'~\cite[II (66), p.~64]{MR0159773}: this is
closed under the injective but non-surjective operation that converts
a thought into the thought that it \emph{is} a thought.  But I will
say with Bob Dylan,
`You don't need a weatherman to know which way the wind blows.'  The
consistency of Dedekind's definition of a simply infinite system is
self-evident.  If we do construct an example, this will tend to
confirm, if anything, the validity of our method of construction,
rather than the consistency of the definition itself.

The existence of a simply infinite system with an underlying
\emph{set} is the Axiom of Infinity of set-theory.  I take set-theory
to begin with the Extension Axiom mentioned in \S\ref{intro}, along
with the \textbf{`Comprehension Axiom':} every property determines a
set, namely the set of objects with the property.  Russell shows this
`axiom' to be false in his letter to Frege:
\begin{myquote}
  Let $w$ be the predicate: to be a predicate that cannot be
  predicated of itself.  Can $w$ be predicated of itself?  From each
  answer its opposite follows.  Therefore we must conclude that $w$ is
  not a predicate.  Likewise there is no class (as a totality) of
  those classes which, each taken as a totality, do not belong to
  themselves.  From this I conclude that under certain conditions a
  definable collection does not form a totality.\hfill
  \cite{Russell-letter} 
\end{myquote}
It is common now to refer to Russell's
 totalities\footnote{Actually the term `totality' in the quotation is a
 translation from Russell's German.} as
\textbf{sets,} and to his definable collections as \textbf{classes.}  Then
  all sets are classes.  The `Comprehension Axiom' is the converse;
  but it is false.  If alternative axioms for set-theory can be
  proposed, and these allow construction of a satisfactory example of
  a simply infinite system, then the construction will not really justify the Peano Axioms; rather, it will tend to justify the set-theory axioms.

Zermelo \cite{Zermelo-invest} offers seven axioms\footnote{Like
  Zermelo, I shall not bother to distinguish axioms from axiom
  \emph{schemes.}} 
 for sets; they can be
expressed as follows:
\begin{sentenceenum}
\item
\textbf{Extension} (as in \S\ref{intro}).
\item
\textbf{Elementary Sets:} classes $\emptyset$, $\{x\}$, and $\{x,y\}$,
with no, one, and two elements, are sets.
\item
\textbf{Separation:}  subclasses of sets are sets.
\item
\textbf{Power Set:}  the class $\pow x$ of subsets of a set $x$ is a set.
\item
\textbf{Union:}  the union $\bigcup x$ of a set $x$ is a set.
\item
\textbf{Choice:}  the union of a set of nonempty disjoint sets has a
subset that has exactly one element in common with each of those
disjoint sets.
\item
\textbf{Infinity:} there is a set that contains $\emptyset$ and closed
under the operation $x\mapsto\{x\}$ of singleton formation.
\end{sentenceenum}
By the Axioms of Union and Elementary Sets, two sets $x$ and $y$ have
a \textsl{union,} $x\cup y$, namely $\bigcup\{x,y\}$.  By the Power Set
Axiom, for every set $x$, there is a set, namely $\pow x$, which has
greater cardinality by Cantor's Theorem.\footnote{The theorem is so
  called by Zermelo \cite[p.~211]{Zermelo-invest}.}  
In particular, the sequence
\begin{align}\label{eqn:pow}
  &x,&& \pow x,& &\pow{\pow x},& &\pow{\pow{\pow x}},& &\dots
\end{align}
is of strictly increasing cardinality, although we have not formulated
a precise notion of sequence.

By the Separation Axiom, the \textsl{intersection} $\bigcap\class C$ of a
nonempty class $\class C$ is a set.  Hence, by the Axiom of Infinity, the
intersection of the class of sets that contain $\emptyset$ and are
closed under $x\mapsto\{x\}$ is a set $\Omega$.  Therefore
$\bigl(\Omega,\emptyset,x\mapsto\{x\}\bigr)$ is a simply infinite system.  In
particular, it admits induction, and we may express this by writing
\begin{equation}\label{Omega}
  \Omega
  =\Bigl\{\emptyset,\{\emptyset\},\bigl\{\{\emptyset\}\bigr\},\dots\Bigr\}. 
\end{equation}
It would however be more satisfactory to obtain $\Omega$, as a class,
\emph{without} assuming the Axiom of Infinity.  To this end,
let us denote by
\begin{equation*}
  \zclass
\end{equation*}
the class comprising every set
whose every element is either $\emptyset$ or $\{x\}$ for some $x$
in the set.
Then $\zclass$ contains the sets
$\emptyset$, $\{\emptyset\}$, $\bigl\{\emptyset,\{\emptyset\}\bigr\}$,
$\Bigl\{\emptyset,\{\emptyset\},\bigl\{\{\emptyset\}\bigr\}\Bigr\}$,
and so forth.  The hope is that $\bigcup\zclass$ is the desired class
$\Omega$. 

To tell whether the hope is realized, we can note first that induction
makes sense for proper classes.  In particular, suppose
$\bigcup\zclass$ has a subclass $\class D$ that contains $\emptyset$
and is closed 
under $x\mapsto\{x\}$.  Then on the class
$\bigcup\zclass\setminus\class D$, the function $\{x\}\mapsto x$ of unique-element extraction is
well-defined, 
and the class is closed under this function.  If we could conclude
that $\bigcup\zclass\setminus\class D$ must be empty---which we
cannot yet do---then
$\bigl(\bigcup\zclass,\emptyset,x\mapsto\{x\}\bigr)$ would
admit induction.  

Skolem \cite[\S4, pp.~296--7]{Skolem-some-remarks} observes that
Zermelo's axioms do
not guarantee the existence of a set comprising the terms of the
sequence in~\eqref{eqn:pow}.  He remedies this with the following
axiom (also discovered by Fraenkel):
\begin{sentenceenum}\setcounter{enumi}{7}
\item
\textbf{Replacement:} the image of a set under a function is a set.
\end{sentenceenum}
Skolem makes the notion of a class precise: it is defined by a
formula in (what we call) the first-order logic of the signature
$\{\in\}$.  These formulas can be defined \textbf{recursively:}  
\begin{sentenceenum}
  \item
Equations $x=y$ and `memberships' $x\in y$ are formulas.
\item
if $x$ is a variable, and $\phi$ and $\psi$ are formulas, then
$\lnot\phi$, 
$(\phi\lto\psi)$, and $\Exists x\phi$ are formulas.
\end{sentenceenum}
More precisely, a formula in \emph{one} free variable defines a class,
while a formula in \emph{two} free variables defines a binary
relation, which may be a function.  A binary relation becomes identified
with a class when, with Kuratowski
\cite[p.171]{Kuratowski},\footnote{An earlier definition of ordered
  pairs is Wiener's \cite{Wiener}:
  $(x,y)=\Bigl\{\bigl\{\{x\},\emptyset\bigr\},\bigl\{\{y\}\bigr\}\Bigr\}$,
  though 
  Wiener calls it
  $\iota`(\iota`\iota`x\cup\iota`\Lambda)\cup\iota`\iota`\iota`y$,
using the notation of Whitehead and Russell \cite[pp.~28--37]{PM}.}  
we define
ordered pairs by the identity
\begin{equation}\label{eqn:op}
  (x,y)=\bigl\{\{x\},\{x,y\}\bigr\}. 
\end{equation}
If we are willing to accept a recursive definition like the definition of formulas in the signature $\{\in\}$, then we should accept a 
definition of natural numbers as strings of the form $\scr\dotsb\scr1$
(or perhaps instead of the form $1\dotsb1$).  However, such definitions do not give us natural
numbers or formulas as composing
\emph{sets.}  Nor is this required for the sake of
formulating set-theory in the first place.

A class on which the function $\{x\}\mapsto x$ is a well-defined
operation is one example of a class whose every element has nonempty
intersection with the class.  Skolem~\cite[\S 6,
  p.~298]{Skolem-some-remarks} perceives the possibility, along with 
the non-necessity, of such
classes: this suggests non-categoricity of the axioms so far.
Von Neumann \cite[\S5, pp.413--4]{von-Neumann-ax} more explicitly
proposes excluding such classes with a new axiom.  His axiom can be
understood as that there is no sequence $(a,a',a'',\dots)$ such
that $a'\in a$, and $a''\in a'$, and so forth.
Since we are now engaged in determining how such a sequence
can be formulated in the first place, let us
introduce the new axiom in the following form.
\begin{sentenceenum}\setcounter{enumi}{8}
\item
\textbf{Foundation:} membership is well-founded on every nonempty
class, that is, some element is disjoint from the
class.\footnote{The Foundation Axiom has a \emph{set} form,
  namely that that membership is well-founded on every nonempty
  \emph{set.}  See Levy~\cite[II.7, p.~72]{MR1924429} on the
  equivalence of the two forms in the presence of the other axioms.
  However, in the absence of the Axiom of Infinity, 
  the set form of Foundation does not imply the class form given
  above.  A proof that this is so can be adapted from a proof given by Cohen~\cite[II.5,
    p.~72]{MR0232676} to show the independence of the set form of
  Foundation from the other axioms.
  Suppose we do have $\vnn$ in the usual sense, as at~\eqref{omega}
  below.  Let $C$ be a set $\{a_x\colon x\in\vnn\}$ indexed by $\vnn$ such that no member is an element of another member.  For example, each $a_n$ could be $\{n+1\}$.  
  If
  $a$ is an arbitrary set, let $\powf x$ denote the set of \emph{finite} subsets
  of $a$.  Now let
  $C^*=\bigcup\{C,\powf C,\powf{\powf C},\dots\}$.
  Define $<$ on
  $C^*$ so that $a<b$ if and only if either 
  $a\in b$, or else $a=a_{k+1}$ and
  $b=a_k$ for some $k$ in $\vnn$.  Then $(C^*,<)$ is a model
  of the first eight numbered axioms above, except Infinity; and it is
  a model of the set form of Foundation.  But it is not a model of the
  class form of Foundation.  Indeed, in $(C^*,<)$, we can define the
  class $\class D$ of all sets $A$ such that, if $b\in A$, then
  $b\included A$.  Then $\class
  D=\bigcup\{\emptyset,\pow{\emptyset},\pow{\pow{\emptyset}},\dots\}$,
  and $C^*\setminus\class D$ is a counterexample to Foundation.} 
\end{sentenceenum}
With the new axiom then, the definition of $\Omega$ in~\eqref{Omega}
as the union
$\bigcup\zclass$ works, in that $\bigcup\zclass$ admits induction
by the argument above.  However, it is not very satisfying to have to
rely on the new axiom.  We can avoid using the axiom as such by
incorporating it into the definition of $\zclass$.  

\begin{theorem}\label{thm:Zm'}
Let $\zclass'$ comprise those elements of $\zclass$ whose every
subset $x$ is disjoint from some element of $x$.  Then
$(\bigcup\zclass',\emptyset,x\mapsto\{x\})$ is a simply infinite system,
even without recourse to the Axioms of Infinity and Foundation.
\end{theorem}

\begin{proof}
If
$A\in\zclass'$, and $\class D$ is a subclass of $\bigcup\zclass'$ that
contains $\emptyset$ and is closed under $x\mapsto\{x\}$, then on 
$A\cap(\bigcup\zclass'\setminus\class D)$, the operation $\{x\}\mapsto x$ is
well-defined, so $A\cap(\bigcup\zclass'\setminus\class D)$ must be empty.  Since
$A$ was arbitrary, $\bigcup\zclass'\setminus\class D$ is empty.
\end{proof}

It may appear that the definition of $\Omega$ as $\bigcup\zclass'$ still does
not yield what is desired.
Foundation excludes sets
on which the operation $\{x\}\mapsto x$ is well-defined; but suppose
$A$ is a nonempty `collection' on which $\{x\}\mapsto x$ is
well-defined.  Then
the union $A\cup\Omega$ may belong to $\zclass'$, unless $A$ can be
obtained as a subclass of this union.  If we can obtain $\Omega$ by
other means, then we can obtain $A$ as
$(A\cup\Omega)\setminus\Omega$; but this just begs the question of whether
we have $\Omega$ as a class in the first place.

If we think we need to fix this problem, we might try letting
$\zclass''$ comprise the \textsl{finite}
elements of $\zclass'$.  Here we can use for example the definition mentioned
by Whitehead and Russell \cite[$*$120.23]{PM-II},\footnote{I learn the
  reference
  from Suppes \cite[p.~102]{MR0349389}; various definitions of
  finite sets are discussed by Tarski \cite{Tarski-finite}.} whereby
$B$ is finite if every
subset of $\pow B$ contains $B$, provided the subset contains $\emptyset$
and is closed under the operation $x\mapsto\{c\}\cup x$ for each $c$
in $B$.  If $\zclass''$ 
had a finite element $A\cup\Omega_0$, where $A$ violated
Foundation, then we could remove the elements
of $\Omega_0$ one by one, thus obtaining $A$, in violation of
Foundation.  However, we still have no guarantee that we can do this,
unless we already have $\Omega_0$ as a class.

The problem here is the one that Skolem hit upon: set-theory is not
categorical. 
Indeed, suppose the
formula $\phi(x)$ defines the class $\Omega$ as in~\eqref{Omega}, and
introduce constants $c$, $c'$, $c''$, \dots  Then the sentences
\begin{align*}
   &\phi(c),  &
 c'&\in c,    &
   &\phi(c'), &
c''&\in c',   &
   &\phi(c''),&
   &\dots
\end{align*}
are consistent with set-theory, in that no contradiction is derivable
from them that is not already derivable from the axioms.  We must
settle for observing that, in any model of set theory, $\phi$ defines
a class that acts \emph{as if} it were as in~\eqref{Omega}.  From
outside the model, we may see an infinite descending chain of elements
of the model; but the chain is not a class of the
model.\footnote{Edward Nelson \emph{recommends} conceiving some
  natural numbers as \textsl{nonstandard} and as composing an infinite
  descending chain.  See Chapter 1, `Internal set theory,' of a
  once-proposed book
  (\url{http://www.math.princeton.edu/~nelson/books.html}, accessed
  March 30, 2009).}

\section{Ordinals}\label{sect:o}

As sets may be gathered into classes, so classes may be gathered
into---let us say \textbf{families.}  For example, given an
equivalence-relation on a class, we have the family of
equivalence-classes of members of that class.  

Given well-ordered sets $\str A$ and $\str B$, let us say $\str A<\str
B$ if $\str A$ is isomorphic to a proper initial segment of $\str B$.
Then exactly one of $x<y$, $x\cong y$, and $y<x$ holds for all
well-ordered sets~$x$ and $y$.  Hence the relation $<$ induces a
well-ordering of the family of isomorphism-classes of well-ordered
sets.

It may be desirable to find a well-ordered class that is isomorphic to the
family of isomorphism-classes of 
  well-ordered sets.  This will be 
  \begin{equation*}
  \on, 
  \end{equation*}
the class of \textbf{ordinal numbers.}
Then every well-ordered set will be isomorphic
  both to a unique ordinal and to the well-ordered set of its
  predecessors in \on.  With this motivation and observation, 
von Neumann~\cite{von-Neumann} obtains \on{} so that each
  element \emph{is} the set of its predecessors.

Suppose $\str A$ is an \textbf{order,} or 
\textbf{ordered set,} $(A,<)$: by
this I mean simply that $<$ is 
irreflexive and transitive on the set $A$.\footnote{So $<$ is what is
  sometimes called a \textsl{strict partial ordering,} although it
  might be total.}  A \textbf{section} of $\str A$ is the set of
predecessors in $A$ of a particular element of $A$.  Let the set of
predecessors of $x$ be denoted by $\pred x$, so that
\begin{equation*}
  \pred x=\{y\in A\colon y<x\}.
\end{equation*}
Then $\str A$
\textbf{admits induction} if the only subset $B$ of $A$ for which
\begin{equation*}
  \pred x\included B\implies x\in B
\end{equation*}
for all $x$ in $A$ is $A$ itself.  There are various ways to define recursion
in ordered sets; it is simplest for present purposes to say that
$\str A$ \textbf{admits recursion} if, for every class $\class C$,
for every function $\cF$ from $\pow{\class C}$ to $\class C$, there is a
unique function $g$ from $A$ to $\class C$ given by
\begin{equation}\label{g}
g(x)=\cF(\{g(y)\colon y<x\})=\cF(g[\pred x]).  
\end{equation}
As with Theorem~\ref{thm:rec}, the following should be standard.  Indeed, possibly awareness of the following has led to the misconception that recursion and induction are equivalent properties of iterative structures.

\begin{theorem}[Recursion]
For \emph{total} orders, admission of
recursion and induction are equivalent to each other and to being
well-ordered. \hfill\qedsymbol
\end{theorem}

If $\str A$ is well-ordered, and
in~\eqref{g} we let $\cF$ be the identity on the universe $\universe$, 
then
\begin{equation}\label{gg}
g(x)=\{g(y)\colon y<x\}=g[\pred x].
\end{equation}
In this case, $(g[A],\in)$ is the ordinal number isomorphic to $\str A$;
indeed, this is so by von Neumann's original \emph{definition} of the
ordinal numbers \cite[p.~348]{von-Neumann}.  Alternative, equivalent
definitions 
include R.~Robinson's \cite{RobinsonR}, namely that an ordinal is a
set that both is \emph{totally} ordered by membership and
\emph{includes} all of its elements.  This 
definition of ordinals assumes the Foundation Axiom, which ensures that each
ordinal is \emph{well}-ordered by membership.  A set that includes all
of its elements is a \textbf{transitive} set; then another definition
of ordinals, again assuming Foundation, is that they are transitive
sets of transitive sets.

The class \on{} has the injective operation $x\mapsto x\cup\{x\}$ of
succession as well as the element~$\emptyset$, which is not a
successor.  Thus \on{} meets two of the three conditions for being a
simply infinite system.  It fails to meet the third condition, of
admitting induction, if there are ordinals other than $\emptyset$ that
are not successors.  Such ordinals are \textbf{limits.}  The
class of ordinals that neither \emph{are} limits nor \emph{contain}
limits is called
\begin{equation}\label{omega}
  \vnn.
\end{equation}
This is an initial segment of \on, so it is either $\on$ itself or a
member of it.  In any case, $(\vnn,\emptyset,x\mapsto
x\cup\{x\})$ is a simply infinite system.  We could let the Axiom of
Infinity be that $\vnn$ is a set.  

Fraenkel \emph{et al.}~\cite[ch.~II, \S3, 6, p.~49]{MR0345816} are vague
on whether this Axiom can be expressed in the form `Such-and-such a
class is a set.'  They list three forms of the
Axiom, called VIa, b, and c, namely: there are sets containing
$\emptyset$ and closed under 
$x\mapsto\{x\}$, $x\mapsto x\cup\{x\}$, and $(x,y)\mapsto x\cup\{y\}$
respectively.  They write:
\begin{myquote}
  In our heuristic classification of the axioms, as to whether they
  are instances of the axiom schema of comprehension or not, the axiom
  of infinity can, to some extent, be viewed as such an instance.
  Each of VIa and VIb can be stated as `there exists a set which
  consists of all natural numbers', for the respective notion of
  natural number, and VIc can be stated similarly.
\end{myquote}
The authors seem to be treating `being a natural number' as a
condition defining a class in the original imprecise sense.  However,
the authors \emph{do} have the precise notion of a condition as being
given by a formula in the first-order logic of $\in$.  I do not
know why they do
not directly address the question of whether, without the Axiom of
Infinity, there is still a formula defining the natural numbers.

The Recursion Theorem for iterative structures (Theorem~\ref{thm:rec}
above)
can be extended to include the following, which gives conditions under
which the equivalence of
properties~\ref{item:ind} and~\ref{item:wo} in \S\ref{no-geo} can be
established. 

\begin{theorem}\label{thm:order}
An iterative structure is a simply infinite system if and only if it can be ordered so that $x<x^{\scr}$ and one of the following holds:
\begin{clauseenum}\renewcommand{\theenumi}{\alph{enumi}}
\item\label{item:wox}
it is well-ordered, and every element besides
$1$ is $x^{\scr}$ for some element $x$, or
\item\label{item:ox}
it admits induction.
\end{clauseenum}
\end{theorem}

\begin{proof}
  Every simply infinite system is isomorphic to
  $\vnn$, which is ordered as indicated so as to satisfy~\eqref{item:wox}.

Suppose $\str B$ is ordered as indicated so as to satisfy~\eqref{item:wox}, and $\str A$ is a
substructure.  If $B\setminus 
A$ were non-empty, then its least element would be $c^{\scr}$ for some
$c$; but $c<c^{\scr}$, so $c\in A$, hence $c^{\scr}\in A$.  Therefore
$B\setminus A$ is empty.  Thus $\str B$ satisfies~\eqref{item:ox}.

Suppose finally $\str A$ is ordered as indicated and satisfies~\eqref{item:ox}.  As noted
in~\S\ref{sect:rec}, addition can be defined in $\str A$ to
satisfy~\eqref{eqn:rec+}.  Again by induction, the ordering of $\str
A$ satisfies $x<x+y$.  Let $h$ be the homomorphism from $(\N,1,\scr)$
to $\str A$.  Then $h$ is surjective by induction in $\str A$, and
$h(x+y)=h(x)+h(y)$ by induction on $y$ in $\N$.  Moreover, in $\N$, if
$x<y$, then $x+z=y$ for some $z$, so 
$h(x)+h(z)=h(y)$ and therefore $h(x)<h(y)$.  Thus $h$ is injective, so
$h$ is an isomorphism, and therefore $\str A$, like $\N$, is a simply infinite system.
\end{proof}

Ordinals suggest the following alternative form of
Theorem~\ref{thm:Zm'}%
, 
that 
is, another way to obtain the class $\Omega$ in~\eqref{Omega}.  

\begin{theorem}\label{thm:sis-wo}
 Let $\class C$ comprise every set
\begin{clauseenum}
  \item
whose every element is
$\emptyset$ or $\{x\}$ for some element $x$ \emph{and} 
\item
which has a
well-ordering $<$ such that $x<\{x\}$ for each element $x$ such that
$\{x\}$ is also an element.  
\end{clauseenum}
Then $(\bigcup\class C,\emptyset,x\mapsto\{x\})$ is a simply infinite
system, even without recourse to the Axioms of Infinity and Foundation.
\end{theorem}

\begin{proof}
  Suppose $A\in\class C$.  Then $\{x\}$ is the successor of $x$ in the
  well-ordering of $A$; that is, there is no $y$ such that
  $x<y<\{x\}$.  Indeed, if there is, let $x$ be least such that there
  is; and then let $y$ be least.  But $y\neq\emptyset$, since
  $\emptyset$ must be the least element of $A$.  Hence $y=\{z\}$ for
  some $z$, and then $z<x$ by minimality of $y$, so $z<x<\{z\}$,
  contrary to the minimality of $x$. 

Now suppose also $B\in\class C$.  Then one of $A$ and $B$ is an 
initial segment of the other, when both are considered as well-ordered sets.  Indeed, one of them, say $A$, is
\emph{isomorphic} to an initial segment of the other, $B$.  If the
isomorphism is $h$, then there can be no least element $x$ of $A$ such
that $h(x)\neq x$, since if there is, then $x=\{y\}$ for some $y$, so
$h(x)=h(\{y\})=\{h(y)\}=\{y\}=x$ (since $\{y\}$ succeeds $y$, and
$\{h(y)\}$ succeeds $h(y)$, in the respective orderings). 

It now follows that $\bigcup\class C$ satisfies
condition~\ref{item:wox} in Theorem~\ref{thm:order}.
\end{proof}

Having \on, we can define the function $\cR$ on \on{} recursively
by\footnote{I use the notation of Kunen~\cite[p.~95]{MR85e:03003}.  It
  seems the original definition is by von Neumann.}
\begin{equation*}
  \cR(\alpha)=\bigcup\{\pow{\cR(\beta)}\colon\beta<\alpha\}.
\end{equation*}
Then $\bigcup\cR[\on]$ is called \wf, the class of
\textbf{well-founded} sets; it 
is a model of the set-theoretic axioms, including Foundation,
regardless of whether the latter `really' holds; so we might as well
assume that it does hold.  Then \wf{} is the class of \emph{all} sets.  

Similarly, there is a function $\cL$ on \on{} such that
$\bigcup\cL[\on]$---the class of \textbf{constructible} sets---is a
model of all of the set-theoretic axioms, including Choice, regardless
of whether the latter `really' holds; so we may assume that it does
hold.\footnote{For the same reason, we may assume that the Generalized
  Continuum Hypothesis holds; but we don't.} 

Now all of the set-theoretic axioms are consistent with those that can
be expressed in the form `Such-and-such a class is a set', along with
Extension. 

\part{}\label{part:new}

\section{Algebras and orderings}\label{new-ord}

Having $\vnn$, we can define arbitrary structures.
This is done at the beginning of any model-theory book (such as
Hodges \cite{MR94e:03002}), except
that these books generally assume that structures and their signatures are
based on \emph{sets.}\footnote{Ziegler \cite{MR1678594} treats the
  `monster model' of a complete theory as based on a proper class; but
  the signature is still a set.}  That assumption is not made here.
  
If $\class C$ is a class, and $n\in\vnn$, then $\class C^n$ can be
understood as the
class of functions from $n$ into $\class C$.  (Such functions are
indeed sets, by the Replacement Axiom, so there can be a class of
them.)  An \textbf{$n$-ary relation} on $\class C$ is a subclass of
$\class C^n$.  In particular, a \textbf{nullary} ($0$-ary) relation is
either $\emptyset$ or $\{\emptyset\}$, that is, $0$ or $1$
(the sets composing $\B$ as used in \S\ref{sect:rec}).  An
\textbf{$n$-ary operation} on $\class C$ is a function 
from $\class C^n$ into $\class C$.  In the most
precise sense, a \textbf{signature} can be defined as two disjoint
classes---of \textbf{predicates} and \textbf{function-symbols}
respectively---together with a function from their union into $\vnn$.  
A signature
$(\sig_{\mathrm p},\sig_{\mathrm f},\arity)$ may be written more
simply as $\sig$, and $\sig$ may be treated as
 $\sig_{\mathrm p}\cup\sig_{\mathrm f}$; then a 
\textbf{structure} with this signature 
is a class $\class C$ together with a function 
$s\mapsto s^{\str C}$ on $\sig$ such
that if $\arity(s)=n$, then $s^{\str C}$ is an $n$-ary relation (when
$s\in\sig_{\mathrm p}$) or
operation (when $s\in\sig_{\mathrm f}$) on $\class C$.  The class
$\class C$ is then the \textbf{universe} of the structure, and
$s^{\str C}$ is the \textbf{interpretation} of $s$ in the structure;
but as noted in~\S\ref{sect:rec}, we may refer to the
interpretation as $s$ also.  I follow the tradition of denoting a
structure by the Fraktur form of the letter used for its universe.  

An \textbf{algebra} is a structure whose signature has
only function-symbols.  If $\str A$ and $\str B$ are algebras with the
same signature, then a homomorphism from the former to the latter is a
function $h$ from $A$ to $B$, inducing the function $h$ from $A^n$ to
$B^n$ for each $n$ in $\vnn$, such that
\begin{equation}\label{eqn:hom}
  h(F^{\str A}(\tuple x))=F^{\str B}(h(\tuple x))
\end{equation}
for all $\tuple x$ in $A^n$, for all $n$-ary $F$ from the signature,
for all $n$ in $\vnn$.  In case $n=0$, \eqref{eqn:hom} becomes
$h(F^{\str A})=F^{\str B}$. 
An algebra \textbf{admits recursion} if from it
to every algebra in the same signature there is a unique
homomorphism; an algebra \textbf{admits induction} if it has no proper
sub-algebra.  
I do not know who first formulated the following generalization of Theorem~\ref{thm:rec}, 
though Enderton \cite[p.~27]{MR0337470} alludes to such a theorem:

\begin{theorem}[Recursion]\label{thm:rec-alg}
An
algebra admits recursion if and only if it admits induction and its
distinguished operations are injective and have disjoint
ranges.\hfill\qedsymbol 
\end{theorem}

In a signature with no nullary function-symbols, the
only algebras that admit induction are empty.
An arbitrary algebra that admits recursion can be called
\textbf{free.}  All free
algebras in the same signature are isomorphic.  

If the universe of an algebra $\str B$ is a set, then $\str B$ has a
subalgebra $\str A$ that admits induction.  Indeed, the universe of
$\str A$ is the intersection of the set
of universes of subalgebras of $\str B$.  In short, $\str A$ is the
smallest subalgebra of $\str B$.
If the
universe of $\str B$ is a proper class, then we cannot obtain $\str A$
in this way.  We might use a trick to obtain $\str A$, as in
Theorem~\ref{thm:Zm'}
or~\ref{thm:sis-wo}.  
Alternatively, if we can obtain a
\emph{free} algebra $\str F$ in the same signature, then $\str
A$ will be the image of $\str F$ in $\str B$.  

In an arbitrary algebraic signature $\sig$, we can obtain a free
algebra as a subalgebra of the algebra $\str S(\sig)$ of
\textbf{strings} of $\sig$.
The universe of $\str S$ is 
$\sig^{<\vnn}$, namely $\bigcup\{\sig^n\colon n\in\vnn\}$; but in this
context an element $k\mapsto s_k$ of $\sig^n$ is written as $s_0\cdots
s_{n-1}$.  The \textbf{length} of every element of $\sig^n$ is $n$. 
If $F$ is an $n$-ary symbol of $\sig$, then $F^{\str S}$
is the operation
\begin{equation}\label{F}
  (t_0,\dots,t_{n-1})\longmapsto\ft 
\end{equation}
of concatenation on $\str S$.  Here the length of $\ft$ is one plus
the sum of the lengths of the $t_k$.  Because the class $\vnn$ of possible
lengths is well-ordered, we have the following.  The theorem is a
requirement for doing symbolic logic, but perhaps not every logician
bothers to work out the proof.

\begin{theorem}
  $\str S(\sig)$ has a subalgebra that is a free algebra.
\end{theorem}

\begin{proof}
Let $\class C$ be the class of subsets $b$ of $\sig^{<\vnn}$ such that
every element of $b$ 
is $Ft_0\cdots t_{n-1}$ for some elements $t_k$ of $b$, for
some $n$-ary $F$ in $\sig$, for some $n$ in $\vnn$.  Let
$A=\bigcup\class C$.  Then $A$ is the universe of a subalgebra
$\str A$ of $\str S(\sig)$, since if $F$ is an $n$-ary element of
$\sig$, and $(t_k\colon k<n)\in A^n$, then each $t_k$ is in some
$b_k$ in $\class C$, and then 
\begin{equation*}
\bigcup\{b_k\colon k<n\}\cup\{Ft_0\cdots t_{n-1}\}\in\class C,
\end{equation*}
so $Ft_0\cdots t_{n-1}$ is in $A$.

The algebra $\str A$ admits induction.  Indeed, if not, then it has a
proper subalgebra with universe $B$.  An element of $A\setminus B$ of
minimal length has the form $\ft$, where each $t_k$
must therefore belongs to $B$; but then $\ft$ must also
be in $B$.

Immediately the operations $F^{\str S}$ have disjoint ranges.

Proper initial substrings of elements of $A$ are not
elements of $A$.  Indeed, suppose this is so for every element of $A$
that is shorter than the element $\ft$.  If $\ft$ has an initial
segment that is in $A$, then this segment must have the form $\fu$,
where one of $t_0$ and $u_0$ is an initial segment of the other, so by
inductive hypothesis they are both equal.  Likewise, if, for some $k$
in $n$, we have that $t_i$ and
$u_i$ are equal when $i<k$, then $t_k=u_k$.  By induction in
$n$, we have $\ft=\fu$.  By induction on lengths, proper initial
segments of elements of $A$ are not elements of $A$.

Consequently each operation $F^{\str S}$ on $A$ is
injective. 
By Theorem~\ref{thm:rec-alg} then,
$\str A$ does indeed admit recursion.
\end{proof}

The free algebra guaranteed by the theorem can be denoted by
\begin{equation*}
  \tm;
\end{equation*}
it comprises the \textbf{closed terms} of $\sig$ in
\textbf{\L ukasiewicz} or
\textbf{`Polish' notation}~\cite[n.~91, p.~38]{MR18:631a}.  For example, suppose $\sig=\{0,\scr,\scrt\}$, where $0$ is nullary, and $\scr$ and $\scrt$ are singulary.  It will be useful to note that the elements of $\tm$ can be arranged in a tree:
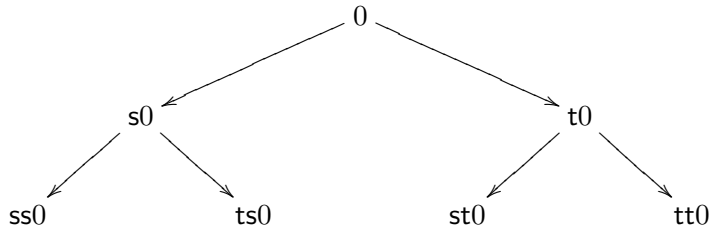
\begin{equation*}
\xymatrix{
&&&0\ar[dll]\ar[drr]&&&\\
&\scr0\ar[dl]\ar[dr]&&&&\scrt0\ar[dl]\ar[dr]&\\
\scr\scr0&&\scrt\scr0&&\scr\scrt0&&\scrt\scrt0
}
\end{equation*}
In the signature $\{0,+\}$, the ordering of $\tm$ is more complicated:
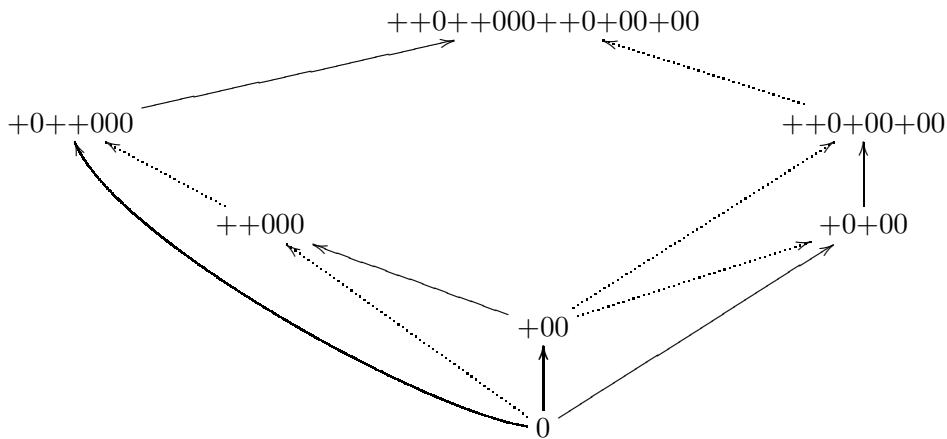
\begin{equation*}
\xymatrix{
&&\pl\pl0\pl\pl000\pl\pl0\pl00\pl00&\\
\pl0\pl\pl000\ar[urr] &&& \pl\pl0\pl00\pl00\ar@{.>}[ul]\\
&\pl\pl000\ar@{.>}[ul]&&\pl0\pl00\ar[u]\\
&&\pl00\ar[ul]\ar@{.>}[uur]\ar@{.>}[ur]&\\
&&0\ar@(l,d)[uuull]\ar@{.>}[uul]\ar[u]\ar@{.>}[u]\ar[uur]
}
\end{equation*}

Again, it was because we already had $\vnn$ as a class that we could
obtain $\tm$ without working out a full analogue of Theorem~\ref{thm:Zm'}
or~\ref{thm:sis-wo}.  However, I wish now, in the remaining sections,
to work out an analogue of
$\vnn$ itself in the signature $\sig$.

\section{Numbers, generalized}

If $\sig$ is an arbitrary algebraic signature, 
we want to obtain a class $\vnns$ that is to $\vnn$ as $\sig$ is to
$\{1,\scr\}$.  We can consider $\vnn$ to arise as follows.  

One starts with the assumption that there is \emph{some} simply infinite system $(\N,1,\scr)$.  With some difficulty, one shows that there is a total ordering of $\N$ such that $x<x^{\scr}$ for all $x$ in $\N$; then one shows that $\N$ is well-ordered by $<$. 
One might for example use the method that Skolem preferred to avoid, as noted in \S~\ref{sect:rec}:
\begin{equation*}
x<y\iff\Exists zx+z=y.
\end{equation*}
Regardless of Skolem's concerns, an infelicity in this definition is its dependence on addition, which in turn depends on whether one treats the initial element of $\N$ as $1$ in the usual sense.  In any case,
 Now one has a function $g$ on $\N$ as in~\eqref{gg}, so that $g(x)=g[\pred x]$.  One then computes
\begin{align}\label{ggg}
&
\begin{aligned}[t]
g(1)
&=g[\emptyset]\\
&=\emptyset,
\end{aligned}
&&
\begin{aligned}[t]
g(x+1)
&=g[\pred{x+1}]\\
&=g[\pred x\cup\{x\}]\\
&=g[\pred x]\cup g[\{x\}]\\
&=g(x)\cup\{g(x)\}.
\end{aligned}
\end{align}
Let $x\cup\{x\}$ be denoted by
\begin{equation*}
x'.
\end{equation*}
Thus $g$ is a homomorphism from $(\N,1,\scr)$ to
$(\universe,\emptyset,{}')$, and the latter has a substructure with
universe $g[\N]$ that admits induction. 
By the Foundation Axiom, the operation $x\mapsto x'$ on $\universe$ is
injective.  Therefore $g$ is an embedding, and $(g[\N],\emptyset,{}')$
is a simply infinite system.   
The ordering of $g[\N]$ induced from $\N$ by $g$ is membership.  By
induction, each element of $g[\N]$ is transitive.  One then defines
$\on$ to consist of the transitive sets that are well-ordered by
membership.  Then $(\on,\emptyset,{}')$ is an iterative structure;
moreover, $\on$ is transitive and well-ordered by membership.  An
element of $\on$ that neither is $\emptyset$ nor belongs to the image
of $x\mapsto x'$ is a \textbf{limit.}  One defines $\vnn$ as the class
of all elements of $\on$ that neither are limits nor contain limits.
Then $\vnn$ contains $\emptyset$ and is closed under $x\mapsto x'$,
but no proper subclass $\class C$ of $\vnn$ is the universe of a
substructure of $(\vnn,\emptyset,{}')$, since the least element of
$\vnn\setminus\class C$ is either $\emptyset$ or else $\alpha'$ for
some $\alpha$ in $\class C$.  Therefore one has recovered $g[\N]$ as
$\vnn$.  In particular, there is no need to use the Peano Axioms to
define an ordering that well-orders $\N$; such an ordering is induced
from $\vnn$. 

Again, as Skolem observed, the ordering of $\N$ can also be defined
recursively, without reference to $\vnn$, by  
\begin{align}\label{n+1}
\pred 1
&=\emptyset,&
\pred{n+1}
&=\pred n\cup\{n\}.
\end{align}
We want similarly to define an ordering on the free algebra $\tm$.
Meanwhile, as a first generalization of~\eqref{ggg}, on $\tm$
we have the 
\textbf{height} function, $\hgt$, given recursively by
\begin{equation}\label{eqn:hgt}
  \hgt(\ft)=\bigcup\bigl\{\hgt(t_k)\cup\{\hgt(t_k)\}\colon k<n\bigr\}.
\end{equation}
By induction, $\hgt(x)$ is in $\vnn$, and
$\hgt(\ft)$ is the greatest of the numbers \mbox{$\hgt(t_k)+1$} (so it is $0$ if $n=0$).
In a generalization of~\eqref{n+1}, we can make the recursive
definitions 
\begin{gather*}
\pred{\ft}=
  \bigcup\{\pred{t_k}\cup\{t_k\}:k<n\},\\
\predk{\ft}=
\begin{cases}
  \pred{t_k}\cup\{t_k\},&\text{ if }k<n,\\
\emptyset,&\text{ if }k\geq n.
\end{cases}
\end{gather*}
Immediately,
\begin{equation}\label{eqn:pred}
  \pred{\ft}=\bigcup\{\predk{\ft}\colon k\in\vnn\}.
\end{equation}
Let us also write
\begin{align}\label{eqn:<<k}
  t<u&\iff t\in\pred u,&t<_ku\iff t\in \predk u.
\end{align}
An ordering \textbf{directs} a class if every finite subset
has an upper bound in the class with respect to the ordering.  An
ordering is \textbf{well-founded} on a class if every section of
the class is a set and every subset of the class has a minimal element.
\begin{lemma}
On $\tm$:
\begin{clauseenum}
\item\label{order-gen-1}
$x<y$ if and only if $x<_ky$ for some $k$ in $\vnn$;
\item\label{order-gen-2}
if $x<y$ and $y<_kz$, then $x<_kz$;
\item \label{order-gen-3}
$<$ and $<_k$ are well-founded orderings;
\item\label{order-gen-4}
$<$ directs $\predk x$;
\item\label{order-gen-5}
$t_k$ is maximal with respect to $<$ in $\predk{\ft}$, assuming $k<n$. 
\end{clauseenum}
\end{lemma}

\begin{proof}
Condition~\eqref{order-gen-1} is a restatement of~\eqref{eqn:pred},
and~\eqref{order-gen-2} is
\begin{equation*}
  y<_kz\implies \pred y\included\predk z,
\end{equation*}
which is established by induction on $z$.  Indeed, suppose the claim
holds when
$z\in\{t_0,\dots,t_{n-1}\}$, but now $y<\ft$.  Then
either $y<_kt_k$ or $y=t_k$ for some $k$ in~$n$.  In the former case, by
inductive hypothesis and~\eqref{order-gen-1}, 
$\pred y\included\predk{t_k}\included\pred{t_k}$.
Hence, in either case,
\begin{equation*}
  \pred y\included\pred{t_k}\included\predk{\ft},
\end{equation*}
so the claim holds when $z=\ft$.

By~\eqref{order-gen-1} and~\eqref{order-gen-2},
the relations $<$ and $<_k$ are transitive.
To complete~\eqref{order-gen-3}, we observe
by induction on $u$ that
\begin{equation*}
  v<u\implies \hgt(v)\in\hgt(u).
\end{equation*}
As $\in$ is irreflexive, so is $<$, and hence so is $<_k$,
by~\eqref{order-gen-1}.  
As $\in$ is well-founded, and the classes $\pred y$ are all sets, $<$
and $<_k$ are well-founded. 

Finally, both~\eqref{order-gen-4} 
and~\eqref{order-gen-5} 
follow from
the observation that, by definition,
$t_k$ is the \emph{maximum} element of $\predk{\ft}$ with respect to $<$, if $k<n$.
\end{proof}

Proving 
the lemma
would take more work if the
corresponding ordering $\in$ of $\vnn$ were not available.  One could
use the ordering on~$\N$ as defined by~\eqref{n+1}; but proving
directly that it is a well-ordering takes more work than proving the
same for $\in$ on $\vnn$.
In any case, we can now prove a generalization of Theorem~\ref{thm:order}.

\begin{theorem}\label{thm:free}
An algebra in a signature $\sig$ is free if and only if
\begin{clauseenum}
  \item\label{free-1}
it has orderings $<$ and $<_k$ for each $k$ in $\vnn$, with
corresponding sets of predecessors as in~\eqref{eqn:<<k},
 such that
\begin{clauseenum}
\item
$x<y$ if and only if $x<_ky$ for some $k$ in $\vnn$,
\item
if $x<y$ and $y<_kz$, then $x<_kz$,
\item
$<$ directs $\predk x$,
\item
$x_k$ is a maximal element with respect to $<$ of
$\predk{F(x_0,\dots,x_{n-1})}$, assuming $k<n$, 
\end{clauseenum}
\item\label{free-2}
the distinguished operations have disjoint ranges, and
\item
one of the following:
\begin{clauseenum}
  \item\label{free-3-a}
the union of these ranges is the whole underlying class of the
algebra, and $<$ is well-founded, or
\item\label{free-3-b}
the algebra admits induction.
\end{clauseenum}
\end{clauseenum}
\end{theorem}

\begin{proof}
If an algebra in $\sig$ is free, then it is isomorphic to $\tm$,
so~\eqref{free-1} follows 
from the lemma.
Also~\eqref{free-2} and~\eqref{free-3-a} follow from this and
Theorem~\ref{thm:rec-alg}. 

Suppose $\str B$ satisfies~\eqref{free-1}, \eqref{free-2},
and~\eqref{free-3-a}, and $\str A$ is a subalgebra.  If $B\setminus A$
is nonempty, then it has a minimal element of the form
$F(x_0,\dots,x_{n-1})$; but then the $x_k$ must be in $A$, and then
so must $F(x_0,\dots,x_{n-1})$ be.  Thus $\str B=\str A$, so it
satisfies~\eqref{free-3-b}. 

Finally, suppose $\str A$ satisfies~\eqref{free-1}, \eqref{free-2},
and~\eqref{free-3-b}.   By Theorem~\ref{thm:rec-alg}, to conclude that
$\str A$ is free, it is enough to show that each $F^{\str A}$ is
injective.  By induction in $\str A$, the homomorphism~$h$ from $\tm$
to $\str A$ is
surjective.  If some $F^{\str A}$ is not injective, then there is some
\emph{minimal} $\ft$ in $\tm$ for which
\begin{equation*}
  F^{\str A}(h(t_0),\dots,h(t_{n-1}))=h(\ft)
=h(\fu)=F^{\str A}(h(u_0),\dots,h(u_{n-1})) 
\end{equation*}
for some $\fu$, although $h(t_k)\neq h(u_k)$ for some $k$.
But $h(t_k)$ 
and $h(u_k)$ are maximal, and the set of them has an upper bound, in $\predk{h(\ft)}$;
hence they are equal.  This contradiction shows $\str A$ is free.
\end{proof}

Towards obtaining $\vnns$ as desired,
suppose $x$ is an $(n+1)$-tuple
$(x_0,\dots,x_n)$ for some $n$ in $\vnn$.  
Let us say that
the \textbf{type} of $x$ is $x_n$; if $k<n$, let us say that the
elements of $x$ of
\textbf{grade} $k$ are the elements of $x_k$.  We may use the following
notation:
\begin{gather*}
y\in_k(x_0,\dots,x_n)\iff k<n\land y\in x_k,\\
\gel(x)=
\begin{cases}
  \bigcup\{x_k\colon k<n\},&\text{ if $x=(x_0,\dots,x_n)$ for some $n$
    in $\vnn$ and some $x_k$},\\
\emptyset,&\text{ otherwise},\\
\end{cases}\\
y\in'x\iff y\in \gel(x).
\end{gather*}
Here $\gel(x)$ is the set of `graded' elements of $x$.  Any elements
of $x_n$, as such, are left out.  

The notion of well-foundedness makes
sense for arbitrary binary relations,\footnote{This observation, with
  the ensuing definition, is apparently due to Zermelo \cite[II.5.1,
    p.~63]{MR1924429}.} such as $\in'$.  Indeed, a binary 
relation $\class R$ is \textbf{well-founded} on a class $\class C$ if
\begin{enumerate}
\item
the class $\{x\colon x\in\class C\land x\mathrel{\class R}a\}$ is a
set whenever $a\in\class C$, 
and 
\item\label{item:minimal}
if $b$ is a nonempty subset of $\class C$, then $b$ has a
\textbf{minimal} element with respect to~$\class R$, that is, an
element $d$ such that 
\begin{equation*}
b\cap\{x\colon x\mathrel{\class R}d\}=\emptyset. 
\end{equation*}
\end{enumerate}
By the axioms of Infinity and Choice, condition~\eqref{item:minimal}
is equivalent to
\begin{enumerate}\setcounter{enumi}1\renewcommand{\labelenumi}{(\theenumi*)}
\item
there is no sequence $(a_n\colon n\in\vnn)$ of elements
of $\class C$ such that $a_{n+1}\mathrel{\class R}a_n$ for each $n$ in
$\vnn$.   
\end{enumerate}
In particular,
well-founded relations are irreflexive.
The Foundation Axiom is just that membership is well-founded on
$\universe$.  If $\class R$ is well-founded on $\class C$, then
$(\class C,\class R)$ \textbf{admits induction} in the sense that the
only subclass $\class D$ of $\class C$ for which 
\begin{equation*}
\class C\cap\{x\colon x\mathrel{\class R}a\}\included\class D
\implies a\in\class D
\end{equation*}
for all $a$ in $\class C$ is $\class C$ itself.

\begin{theorem}\label{thm:uuu}
The relation $\in'$ is well-founded on $\universe$.
\end{theorem}

\begin{proof}
  Using the definition~\eqref{eqn:op} for ordered pairs,
we have
\begin{gather*}
  (x_0,\dots,x_n) =\Bigl\{\bigl\{\{0\},\{0,x_0\}\bigr\},\dots,
  \bigl\{\{n\},\{n,x_n\}\bigr\}\Bigr\},\\
\bigcup\bigcup(x_0,\dots,x_n)=\{0,\dots,n,x_0,\dots,x_n\},
\end{gather*}
so $\{y\colon y\in'x\}\included\bigcup\bigcup\bigcup x$, a set.
Suppose there were a sequence $(x_n\colon n\in\vnn)$ such that always
$x_{n+1}\in'x_n$.  Then always $x_{n+1}\in\bigcup\bigcup\bigcup x_n$.
But then, assuming $x_n\in\cR(\alpha_n)$, we should have also
$\bigcup\bigcup\bigcup x_n\in\cR(\alpha_n)$, and then 
 $x_{n+1}\in\cR(\alpha_{n+1})$ for some $\alpha_{n+1}$ that was
strictly less than $\alpha_n$.  There is no such sequence
$(\alpha_n\colon n\in\vnn)$ of ordinals. 
\end{proof}

For a better generalization of~\eqref{ggg} than~\eqref{eqn:hgt}, we
make $\universe$ into an $\sig$-algebra by defining
\begin{equation}\label{FV}
  F^{\universe}(x_0,\dots,x_{n-1})
  =(\gel(x_0)\cup\{x_0\},\dots,\gel(x_{n-1})\cup\{x_{n-1}\},F)
\end{equation}
for all $F$ in $\sig$.

\begin{theorem}\label{thm:V-alg}
  The operations $F^{\universe}$ are injective and have disjoint
  ranges. 
\end{theorem}

\begin{proof}
  That the ranges are disjoint is immediate from the definitions.  If
  \begin{equation*}
  F^{\universe}(x_0,\dots,x_{n-1})= 
  F^{\universe}(y_0,\dots,y_{n-1}), 
  \end{equation*}
  but $x_k\neq y_k$ for some $k$,
  then, since
  \begin{equation*}
\gel(x_k)\cup\{x_k\}=
\gel(y_k)\cup\{y_k\},
  \end{equation*}
we have $x_k\in \gel(y_k)$ and $y_k\in \gel(x_k)$, that is, $x_k\in'y_k$ and
$y_k\in'x_k$.  But this contradicts Theorem~\ref{thm:uuu}.  
\end{proof}

Considering $\universe$ as an $\sig$-algebra, we can now define
\begin{equation*}
  \vnns
\end{equation*}
as the homomorphic
image of $\tm$ in $\universe$.
Then $\vnns$ is free by Theorems~\ref{thm:rec-alg}
and~\ref{thm:V-alg}.
But I propose to obtain $\vnns$ alternatively as a certain subclass of a
class $\ons$, just as $\vnn$ is obtained from $\on$.

\section{Ordinals, generalized}

Let us define
\begin{equation*}
\predk x=\{y\colon y\in_k x\},
\end{equation*}
so that
\begin{equation*}
\predk{(x_0,\dots,x_n)}=
\begin{cases}
x_k,&\text{ if }k<n,\\
\emptyset,&\text{ otherwise,}
\end{cases}
\end{equation*}
and
\begin{equation*}
\gel(x)=\bigcup\{\predk x\colon k\in\vnn\}.
\end{equation*}
Let us say that $x$ is \textbf{$k$-transitive} if
\begin{equation*}
  y\in\predk x\implies \gel(y)\included\predk x.
\end{equation*}
In order to define $\ons$, we first define a class $\Ds$, which
comprises all $x$ such that 
\begin{clauseenum}
  \item\label{d1}
$x$ is an $(n+1)$-tuple having the type of an $n$-ary element of
    $\sig$ for some $n$ in $\vnn$, and then
$\predk x\neq\emptyset$ whenever $k<n$;
\item\label{d2}
each element $y$ of $\gel(x)$ is an $(m+1)$-tuple having the type of
an $m$-ary element of $\sig$ for some $m$ in $\vnn$, and then $\predl
y\neq\emptyset$ whenever $\ell<m$;
\item\label{d3}
$x$ is $k$-transitive for each $k$ in $\vnn$;
\item\label{d4}
each element of $\gel(x)$ is $k$-transitive for each $k$ in $\vnn$. 
\end{clauseenum}
We shall presently define $\ons$ as a subclass of $\Ds$.  Meanwhile,
we already have some analogues to properties of $\on$.

\begin{lemma}\label{lem:in'-trans}
  The relation $\in'$ is transitive on $\Ds$.
\end{lemma}

\begin{proof}
  If $\Ds$ contains $x$, $y$, and $z$, where $y\in'x$ and $z\in'y$,
  then $y\in\gel(x)$, so $y\in\predk x$ for some $k$, and therefore
  $\gel(y)\included\predk x\included\gel(x)$ by~\eqref{d3}; but 
  also $z\in\gel(y)$, so $z\in\gel(x)$, that is, $z\in'x$.
\end{proof}

Also, like $\on$ itself, $\Ds$ has a kind of transitivity: 

\begin{lemma}\label{lem:Ds-trans}
  If $x\in\Ds$, then $\gel(x)\included\Ds$.
\end{lemma}

\begin{proof}
  Suppose $x\in\Ds$ and $y\in\gel(x)$.  Then $y$ satisfies~\eqref{d1},
  by~\eqref{d2} for $x$.  Also $y$ satisfies~\eqref{d3}, by~\eqref{d4}
  for $x$.  Finally, $y\in\predk x$ for some $k$ in $\vnn$, so
  by~\eqref{d3} for $x$ we have $\gel(y)\included\predk x$, hence
  $\gel(y)\included\gel(x)$.  So $y$ satisfies~\eqref{d2}
  and~\eqref{d4}.  Therefore $y\in\Ds$.
\end{proof}

When $\sig=\{1,\scr\}$, then $\Ds$ is not really anything new.
Indeed, informally,
\begin{equation*}
  \Dss=\Bigl\{
(1),
\bigl(\bigl\{(1)\bigr\},\scr\bigr),
\Bigl(\Bigl\{(1),\bigl(\bigl\{(1)\bigr\},\scr\bigr)\Bigr\},\scr\Bigr),
\dots\Bigr\}.
\end{equation*}
Using the definition in~\eqref{FV}, we have:
\begin{theorem}
$(\Dss,1^{\universe},\scr^{\universe})\cong
  (\on,\emptyset,\alpha\mapsto\alpha+1)$.    
\end{theorem}

\begin{proof}
Note that $1^{\universe}$ is the $1$-tuple $(1)$.
  From $\on$ to $\universe$, there is a function $\class H$ defined
  recursively by
  \begin{equation*}
\class H(\alpha)=
  \begin{cases}
(1),& \text{ if }\alpha=\emptyset;\\
(\class H[\alpha],\scr),& \text{ if }\alpha\neq\emptyset.
  \end{cases}
  \end{equation*}
Then 
\begin{equation}\label{eqn:GH}
\gel(\class H(\alpha))=\predo{\class H(\alpha)}=\class H[\alpha].
\end{equation}
Therefore
\begin{equation*}
  \class H(\alpha+1)
=(\class H[\alpha]\cup\{\class H(\alpha)\},\scr)
=(\gel(\class H(\alpha))\cup\{\class H(\alpha)\},\scr)
=\scr^{\universe}(\class H(\alpha)),
\end{equation*}
which means $\class H$ is a homomorphism of $\{1,\scr\}$-algebras.  If
$\alpha<\beta$, then 
$\class H(\alpha)\in\class H[\beta]$, 
so $\class H(\alpha)\in'\class H(\beta)$ by~\eqref{eqn:GH}.  Therefore
$\class H$ is order-preserving and, in particular, injective.  It remains to show that the range of $\class H$ is $\Dss$.

In the definition of $\Dss$, conditions~\eqref{d1} and~\eqref{d2}
hold for elements of $\class H[\on]$ by definition of $\class H$.
Then~\eqref{d3} follows; that is,
$\class H(\beta)$ is always $k$-transitive.  Indeed, it is
trivially so, if $\beta=0$ or $k>0$; and by~\eqref{eqn:GH}, we have
that, if $x\in\predo{\class
  H(\beta)}$, then $x=\class H(\alpha)$ for some $\alpha$, so
\begin{equation*}
\gel(x)=\class H[\alpha]\included\class H[\beta]=\predo{\class
  H(\beta)}.
\end{equation*}
Since each element of $\gel(\class H(\beta))$ is some $\class
H(\alpha)$, it is $k$-transitive too: thus~\eqref{d4}.  So
$\class H$ maps $\on$ into $\Dss$.

By induction in $\Dss$ with respect to the well-founded relation
$\in'$, we establish
$\Dss\included\class H[\on]$.  Suppose $x\in\Dss$.  Then
$\gel(x)\included\Dss$ by the lemma.
As an inductive hypothesis, suppose $\gel(x)\included\class H[\on]$.
If $\class H(\beta)\in\gel(x)$, 
and $\alpha<\beta$, then, since $\class H(\alpha)\in\gel(\class
H(\beta))$ by~\eqref{eqn:GH}, we have
$\class H(\alpha)\in\gel(x)$ by the
$0$-transitivity of $x$.  Thus $\{\alpha\colon\class
H(\alpha)\in\gel(x)\}$ is transitive, so it is an ordinal $\beta$, and
$\class H[\beta]=\gel(x)$.  If $\beta=\emptyset$, then $x=(1)=\class
H(\beta)$ by the second part of condition~\eqref{d1} in the definition
of $\Dss$; if
$\beta\neq\emptyset$, then $\gel(x)\neq\emptyset$, so $x$ can only be
$(\gel(x),\scr)$, which is also $\class H(\beta)$.  So $\class H$ is
an isomorphism between $\on$ and $\Dss$.
\end{proof}

An element $x$ of $\Ds$ can be called a \textbf{limit} if some set
$\predk x$ has no maximal element with respect to
$\in'$.  In general, there may be $x$ in $\Ds$ such that neither $x$
nor any element of $\gel(x)$ is a limit, but still $x$ is not in
$\vnns$.  Indeed, if $\sig=\{a,b,\scr\}$, where $a$ and $b$
are nullary, and $\scr$ is singulary as usual, then $\Dss$ contains
$\bigl(\bigl\{(a),(b)\bigr\},\scr\bigr)$, which is not in $\vnns$.
So we let
\begin{equation*}
  \ons
\end{equation*}
denote the class of $x$ in $\Ds$ that meet the additional conditions
\begin{clauseenum}\setcounter{enumi}{4}
  \item\label{d5}
$\in'$ directs each set $\predk x$;
\item\label{d6}
$\in'$ directs each set $\predl y$ for each $y$ in $\gel(x)$.
\end{clauseenum}
We already know from Theorem~\ref{thm:uuu} that $\in'$ is well-founded
on $\ons$ and its elements; but this relies on the axioms of
Foundation, Infinity, and Choice. 
To avoid this reliance, we can impose one additional condition on the
elements $x$ of $\ons$:
\begin{clauseenum}\setcounter{enumi}{6}
  \item\label{d7}
$\in'$ is well-founded on $\gel(x)$.
\end{clauseenum}
As will be shown with Theorem~\ref{thm:vnns-in-ons} below, this definition of $\ons$
is enough to ensure that distinguishing $\vnns$ as a subclass of $\ons$ is analogous to distinguishing $\vnn$ as a subclass of $\on$.  Meanwhile, we can work out a particular example, as in the next section.

\section{Trees}

Suppose $\sig=\{0,\scr,\scrt\}$, where $0$ is nullary, and $\scr$ and $\scrt$ are singulary.  

Let $x\in\ons$.  Then $\predo x=\gel(x)$.  Consequently, $\gel(x)$ is well-ordered by $\in'$.  Indeed, to establish this, by Lemma~\ref{lem:in'-trans} and~\eqref{d7}, it is enough to note that $\in'$ is total on $\gel(x)$.  Suppose on the contrary that $\gel(x)$ has incomparable elements $y$ and $z$.  Then $y,z\in'a_0$ for some $a_0$ in $\gel(x)$, by~\eqref{d5}.  But then $y,z\in'a_1$ for some $a_1$ in $\gel(a_0)$, by~\eqref{d6}.  But $a_1\in\gel(x)$ by~\eqref{d3}; so $y,z\in'a_2$ for some $a_2$ in $\gel(a_1)$, and so on.  This violates~\eqref{d7}.

Since each 

\section{Numbers as ordinals}

\begin{theorem}\label{thm:in'-wf}
  The relation $\in'$ is a well-founded ordering of $\ons$.
\end{theorem}

\begin{proof}
  By Lemma~\ref{lem:in'-trans}, $\in'$ is transitive on $\ons$.  If
  $x\in\ons$, then, because $\in'$ is well-founded on $\gel(x)$
  by~\eqref{d7}, in particular $x\notin'x$.  Therefore $\in'$ orders
  $\ons$.
Also $\{y\colon y\in\ons\land y\in'x\}$ is a set, being a subclass of
$\gel(x)$. 
  Suppose $b\included\ons$ and $x\in b$.  If $x$ is not a
  minimal element of $b$ with respect to $\in'$, then $b\cap\{y\colon
  y\in'x\}$, being a subset of $\gel(x)$, has a minimal element; but
  this minimal element is minimal in $b$ as well, by transitivity of
  $\in'$.  Therefore $\in'$ is well-founded on $\ons$.
\end{proof}

\begin{theorem}\label{thm:ons-trans}
  If $x\in\ons$, then
  $\gel(x)\included\ons$.
\end{theorem}

\begin{proof}
Suppose $x\in\ons$ and $y\in \gel(x)$.  By Lemma~\ref{lem:Ds-trans},
we know $y\in\Ds$.
By~\eqref{d6} for $x$, we know $y$ satisfies~\eqref{d5}; by~\eqref{d3}
also, $y$ satisfies~\eqref{d6}.  Finally, since
$\gel(y)\included\gel(x)$ by~\eqref{d3}, we have~\eqref{d7} for $y$.
\end{proof}

\begin{theorem}
  $\ons$ is an $\sig$-subalgebra of $\universe$.
\end{theorem}

\begin{proof}
Suppose $x_k\in\ons$ when $k<n$, and let
$y=F^{\universe}(x_0,\dots,x_{n-1})$.  We show that $y$ satisfies the defining
conditions of $\ons$.  
\begin{asparaenum}[1.]
\item
First, $y$ satisfies~\eqref{d1}
by definition, that is,~\eqref{FV}.  
\item
Suppose $z\in \gel(y)$; equivalently, $z\in\predk y$ for some
$k$ in $n$.  But
\begin{equation}\label{eqn:predky}
  \predk y=\gel(x_k)\cup\{x_k\},
\end{equation}
so in particular either $z=x_k$ or $z\in \gel(x_k)$.  By~\eqref{d1}
and~\eqref{d2} respectively for $x_k$, we have that $y$
satisfies~\eqref{d2}.  
\item
Also, if $z=x_k$, then
\begin{equation*}
\gel(z)=\gel(x_k)\included\predk y,
\end{equation*}
while if $z\in \gel(x_k)$, then $z\in\predl{x_k}$ for some $\ell$, so
that
\begin{equation*}
  \gel(z)\included\predl{x_k}\included \gel(x_k)\included\predk y
\end{equation*}
by~\eqref{d3} for $x_k$ and~\eqref{eqn:predky}.  Thus~$y$ satisfies~\eqref{d3}.
\item
If $z=x_k$, then $z$ is $\ell$-transitive for each $\ell$ in $\vnn$
by~\eqref{d3} for $x_k$, while if $z\in \gel(x_k)$, then $z$ is
$\ell$-transitive for each $\ell$ in $\vnn$ by~\eqref{d4} for
$x_k$.  Therefore~$y$ satisfies~\eqref{d4}, and so $y\in\Ds$.
\item
By~\eqref{eqn:predky}, we have that $x_k$ is the greatest
element of $\predk y$ with respect to $\in'$, so this ordering directs
$\predk y$. 
Thus $y$ satisfies~\eqref{d5}.  
\item
If $z=x_k$, then $\in'$ directs
$\predl z$ by~\eqref{d5} for $x_k$, while if $z\in \gel(x_k)$, then
$\in'$ directs $\predl z$ by~\eqref{d6} for $x_k$.  Therefore $y$
satisfies~\eqref{d6}.
\item
We have $\gel(y)=\bigcup\{\gel(x_k)\cup\{x_k\}\colon k<n\}$.  By
Theorems~\ref{thm:in'-wf} and~\ref{thm:ons-trans}, since $\in'$ is
well-founded on each $\gel(x_k)$, it is also well-founded on each set
$\gel(x_k)\cup\{x_k\}$ and hence on their union.
So $y$ is in $\ons$.\qedhere
\end{asparaenum}
\end{proof}

Now $\vnns$ can be understood as the image of $\tm$ in $\ons$.  But
again, we want to obtain $\vnns$ independently from $\tm$.

\begin{theorem}\label{thm:vnns-in-ons}
  $\vnns$ consists of those $x$ in $\ons$ such that neither $x$ nor any
  element of $\gel(x)$ is a limit.
\end{theorem}

\begin{proof}
By Theorem~\ref{thm:in'-wf}, we can argue by induction.
Suppose $x_k$ meets the 
  conditions when $k<n$.  Let $z=F^{\universe}(x_0,\dots,x_{n-1})$.  Each $x_k$
  is a maximal element of $\predk z$, so $z$ is not a limit.  Each
  element of $\gel(z)$ is either $x_k$ or an element of $\gel(x_k)$ for some
  $k$ in $n$, so it is not a limit either.  Thus $z$ meets the
  conditions.

To prove the converse, suppose if possible that
$(x_0,\dots,x_{n-1},F)$ is a counterexample that is minimal with
respect to $\in'$.  Then each of the $x_k$ has a maximal element,
$u_k$, with respect to $\in'$.  But $\in'$ directs $x_k$, so $u_k$ is
the greatest element, hence $x_k\included \gel(u_k)\cup\{u_k\}$.  Also
$\gel(u_k)\included x_k$ by the
$k$-transitivity of $(x_0,\dots,x_{n-1},F)$.  Thus
\begin{equation*}
  (x_0,\dots,x_{n-1},F)=(\gel(u_0)\cup\{u_0\},\dots,\gel(u_{n-1})\cup\{u_{n-1}\},F)
=F^{\universe}(u_0,\dots,u_{n-1}).
\end{equation*}
But also, neither $u_k$ nor
any element of $\gel(u_k)$ is a limit, so $u_k\in\vnns$ by minimality
of $(x_0,\dots,x_{n-1},F)$, and
therefore $F^{\universe}(u_0,\dots,u_{n-1})\in\vnns$. 
\end{proof}


\def\cprime{$'$} \def\cprime{$'$} \def\cprime{$'$}
\providecommand{\bysame}{\leavevmode\hbox to3em{\hrulefill}\thinspace}
\providecommand{\MR}{\relax\ifhmode\unskip\space\fi MR }
\providecommand{\MRhref}[2]{%
  \href{http://www.ams.org/mathscinet-getitem?mr=#1}{#2}
}
\providecommand{\href}[2]{#2}

\end{document}